\documentclass[11pt]{article}

\usepackage[T1]{fontenc}
\usepackage[utf8]{inputenc}
\usepackage{lmodern}
\usepackage{microtype}
\usepackage{amsmath,amssymb,amsthm,mathtools,enumerate}
\usepackage{enumitem}
\usepackage{geometry}
\usepackage{hyperref}
\usepackage{aliascnt}
\usepackage[nameinlink,noabbrev]{cleveref}

\hypersetup{
  colorlinks=true,
  linkcolor=black,
  citecolor=black,
  urlcolor=black,
  pdftitle={Gallai Decomposition of Ordered Groups: Subgroups, Quotients, and the N-free Case},
  pdfauthor={Imed Zaguia}
}

\newtheorem{theorem}{Theorem}[section]
\numberwithin{equation}{section}
\newaliascnt{lemma}{theorem}
\newtheorem{lemma}[lemma]{Lemma}
\aliascntresetthe{lemma}
\newaliascnt{proposition}{theorem}
\newtheorem{proposition}[proposition]{Proposition}
\aliascntresetthe{proposition}
\newaliascnt{corollary}{theorem}
\newtheorem{corollary}[corollary]{Corollary}
\aliascntresetthe{corollary}
\theoremstyle{definition}
\newaliascnt{definition}{theorem}
\newtheorem{definition}[definition]{Definition}
\aliascntresetthe{definition}
\newaliascnt{example}{theorem}
\newtheorem{example}[example]{Example}
\aliascntresetthe{example}
\theoremstyle{remark}
\newaliascnt{remark}{theorem}

\aliascntresetthe{remark}
\newaliascnt{observation}{theorem}

\aliascntresetthe{observation}

\crefname{theorem}{Theorem}{Theorems}
\crefname{lemma}{Lemma}{Lemmas}
\crefname{proposition}{Proposition}{Propositions}
\crefname{corollary}{Corollary}{Corollaries}
\crefname{definition}{Definition}{Definitions}
\crefname{example}{Example}{Examples}
\crefname{remark}{Remark}{Remarks}
\crefname{observation}{Observation}{Observations}

\newcommand{\e}{\mathsf e}
\newcommand{\inc}{\operatorname{inc}}
\newcommand{\Comp}{\operatorname{Comp}}
\newcommand{\Inc}{\operatorname{Inc}}
\newcommand{\supp}{\operatorname{supp}}
\newcommand{\width}{\operatorname{width}}
\newcommand{\Ctype}{\mathbf C}
\newcommand{\Atype}{\mathbf A}
\newcommand{\cV}{\mathcal V}
\newcommand{\cH}{\mathcal H}

\title{Gallai Decomposition of Ordered Groups: Subgroups, Quotients, and the \(N\)-free Case}
\author{%
Imed Zaguia\thanks{The author gratefully acknowledges support from the Canadian Defence Academy.}\\
Department of Mathematics \& Computer Science\\
Royal Military College of Canada\\
P.O.Box 17000, Station Forces\\
Kingston, Ontario, Canada\\
e-mail: zaguia@rmc.ca%
}
\date{\today}

\begin{document}
\maketitle

\begin{abstract}
We study Gallai decomposition for groups equipped with two-sided invariant
partial orders.  The key algebraic step extends to arbitrary binary relations
compatible with the group operation: if all left and right translations
preserve a binary relation \(\rho\), then every least strong module
\(S_\rho(\e,g)\), \(g\ne\e\), is a subgroup.  For a partial order this
subgroup is convex.  Thus the robust modules through the identity of an
ordered group form a canonical chain of convex subgroups, with each
canonical factor \(H/H^-\) prime, totally ordered, or equality-ordered.  We
characterize exactly the subgroups that are modules, show that they form a
complete sublattice of the subgroup lattice, establish overlap and inheritance
results for arbitrary subgroups, and prove compatibility with quotients by
normal strong subgroups.

For \(N\)-free ordered groups the prime factors disappear.  Using the
robust-module decomposition of cographs, we characterize all two-sided
invariant \(N\)-free partial orders by reduced admissible two-coloured subgroup
chains, with totally ordered and equality-ordered canonical factors; the order
is determined by the first nontrivial factor of each element.  We also
determine how the canonical decomposition restricts to arbitrary subgroups,
characterize finite width and prove width divisibility for subgroups, and show
that every reduced two-coloured chain is realized by an \(N\)-free ordered
abelian group whose canonical factors are isomorphic to \(\mathbb Z\).
\end{abstract}

\medskip
\noindent\textbf{2020 Mathematics Subject Classification.}
Primary 06F15; Secondary 06A07, 05C75, 20F60.


\smallskip
\noindent\textbf{Keywords.} partially ordered group; Gallai decomposition; compatible binary relation; module; Cayley graph; \(N\)-free poset; cograph.

\section{Introduction}

Modular decomposition is a basic structural tool for partially ordered sets,
comparability graphs, and more general binary structures
\cite{Gallai1967,Kelly1985,EhrenfeuchtRozenbergII,HarjuRozenberg1994,
ille-woodrow-2009,CourcelleDelhomme2008,BoudabbousDelhomme2012}.
For a binary relation \(\rho\) on a set \(X\), a subset
\(M\subseteq X\) is a \emph{module} if, for every
\(x\in X\setminus M\) and every \(y,z\in M\),
\[
 x\rho y \iff x\rho z
 \qquad\text{and}\qquad
 y\rho x \iff z\rho x.
\]
The empty set, the singletons of \(X\), and \(X\) itself are modules,
called \emph{trivial}.  The structure is \emph{prime} if it has no
nontrivial module.  A module is \emph{strong} if it overlaps no other
module.

For \(x,y\in X\), let \(S_\rho(x,y)\)
denote the least strong module containing \(x\) and \(y\).  These
\emph{robust} modules are precisely the pair-generated nodes of the
strong-module hierarchy; together with their quotient structures, they
encode the modular decomposition
\cite{CourcelleDelhomme2008,BoudabbousDelhomme2012}.  Thus
\(S_\rho(x,y)\) are canonical, rather than auxiliary, objects of the
decomposition.

For arbitrary binary relations, Gallai decomposition expresses each
nonsingleton robust module as a lexicographical sum of its maximal proper
strong submodules, and the corresponding quotient has no nontrivial strong
modules; hence it is prime, edge-free, complete, or a total order.
Specializing to posets, the quotient is therefore prime, an antichain, or
a total order.  When the relation is clear, we write simply \(S(x,y)\).

We determine how this decomposition interacts with a group operation.  An
\emph{ordered group} is a group \(G\) equipped with a partial
order invariant under multiplication on both sides; see, for example,
\cite{Fuchs1963,Glass1999}.  We write the group multiplicatively and denote
its identity by \(\e\).  Left and right translations are order
automorphisms, while inversion is an order anti-automorphism.  Module
subsets and lexicographical quotient decompositions have already proved
useful in the study of ordered groups \cite{PouzetZaguia2019}.  Our focus is
the \emph{canonical Gallai decomposition} of the underlying
ordered set: how the group symmetries constrain it, and, once the prime
Gallai alternative is excluded, which chains of chain and antichain factors
can arise from a two-sided invariant group order.

The canonical Gallai objects acquire algebraic structure by an argument not
specific to partial orders.  In \cref{thm:robust-subgroup} we prove that if a
group carries a binary
relation \(\rho\) preserved by all left and right translations, then \(S_\rho(\e,g)\)
is a subgroup for every \(g\ne\e\).  Thus the pair-generated nodes on the
identity branch of the ordinary modular decomposition are themselves
subgroups.  Specializing to a partial order, they are convex subgroups.

There are related subgroup phenomena in the literature on Cayley graphs
and digraphs.  Peng and Wang characterize nontrivial lexicographic
decompositions of finite Cayley digraphs by a double-coset condition on a
subgroup \cite{PengWang2007}.  Barber develops the corresponding
subgroup--coset viewpoint for wreath decompositions of Cayley digraphs
and, more generally, double-coset digraphs \cite{Barber2021}.  More
recently, Chudnovsky et al.\ study modules in finite Cayley graphs and show,
in particular, that when the graph is connected and co-connected, a maximal
module containing the identity is a subgroup \cite{ChudnovskyEtAl2026}.

Our point of departure is instead the canonical modular decomposition: the
intrinsically defined least strong modules \(S_\rho(\e,g)\), rather than
a prescribed block system, a chosen homogeneous set, or the prior existence
of a lexicographic or wreath decomposition.  The resulting subgroups are
therefore canonical nodes of the modular decomposition.

For a Cayley relation
\[
   x\,\rho_D\,y
   \quad\Longleftrightarrow\quad
   x^{-1}y\in D,
\]
two-sided compatibility is equivalent to conjugation invariance of
\(D\).  Hence the relation-theoretic result applies, in particular, to
conjugation-invariant Cayley graphs and digraphs and to every Cayley
relation on an abelian group.

For an ordered group, put \(\cV(G):=\{S(\e,g):g\ne\e\},\)
and, for \(H\in\cV(G)\), \(H^-:=\{\e\}\cup\bigcup\{K\in\cV(G):K<H\}.\)
The robust modules through \(\e\) form one branch of the modular hierarchy,
and the group operation turns this branch into a canonical subgroup chain.

\begin{theorem}[Gallai decomposition of ordered groups]
\label{thm:intro-gallai}
Let \(G\) be an ordered group.
\begin{enumerate}[label=\textup{(\alph*)}]
\item
Every strong module of \(G\) containing \(\e\) is a convex subgroup of \(G\),
and \(\cV(G)\) is a chain under inclusion.  More precisely, the strong
modules containing \(\e\) are exactly
\[
 H_I:=\{\e\}\cup\bigcup_{H\in I}H,
\]
where \(I\) is an initial segment of \(\cV(G)\).  Consequently every strong
module of \(G\) is a left coset of such a convex subgroup.

\item
For every \(H\in\cV(G)\), the subgroup \(H^-\) is proper, convex and normal
in \(H\).  Its cosets are the maximal proper strong submodules of \(H\), and
the quotient ordered group
\[
 Q_H:=H/H^-
\]
is prime, totally ordered, or equality-ordered.

\item
If \(g\ne\e\) and \(H=S(\e,g)\), then the comparison of \(g\) with \(\e\)
in \(G\) is determined by \(gH^-\) in \(Q_H\).  Hence the order relation on
\(G\) is determined by the canonical chain together with its ordered Gallai
quotients, and conjugation preserves the canonical decomposition of \(G\).
\end{enumerate}
\end{theorem}

The decomposition is also stable under the natural quotient operation.  If
\(K\triangleleft G\) is a proper strong subgroup containing \(\e\), then
\cref{cor:quotient-compatibility} gives
\[
 \cV(G/K)=\{H/K:H\in\cV(G),\ K<H\},
 \qquad
 S_{G/K}(K,gK)=S_G(\e,g)/K,
\]
and each canonical factor above \(K\) is unchanged up to its natural ordered
group isomorphism.

The canonical chain also constrains arbitrary subgroups.  If \(K\leq G\),
then the intersections \(K\cap H\), \(H\in\cV(G)\), form a chain of convex
subgroup-modules of the induced order on \(K\), and each nontrivial factor \(\frac{K\cap H}{K\cap H^-}\)
embeds as an ordered subgroup of \(Q_H\).  More generally, if \(H\) is a
subgroup which is a module of \(G\) and \(K\leq G\) overlaps \(H\), then
\(L=H\cap K\) is a convex subgroup-module of \(K\) and
\[
 K\cap kH=kL,\qquad K\cap Hk=Lk\qquad(k\in K).
\]
Thus a subgroup crossing a modular layer meets every coset that it encounters
in a translate of one fixed subgroup.

This leads to an exact description of the subgroup-modules.  The local
picture also explains where their lattice can branch: prime factors admit no
proper nontrivial subgroup-module, total factors admit only convex ones,
while an equality-ordered factor admits every subgroup.  We write
\(\operatorname{Sub}(Q)\) for the subgroup lattice of \(Q\) and
\(\operatorname{ConvSub}(Q)\) for its convex-subgroup lattice when \(Q\) is
ordered.

\begin{theorem}[Subgroup-modules]
\label{thm:intro-subgroup-modules}
Let \(G\) be an ordered group and let \(K\leq G\).  Then \(K\) is a
module of \(G\) if and only if either
\begin{enumerate}[label=\textup{(\alph*)}]
\item \(K\) is the strong subgroup determined by an initial segment of
\(\cV(G)\); or
\item there is a unique \(H\in\cV(G)\) such that
\[
 H^-<K<H
\]
and \(K/H^-\) is a subgroup-module of \(Q_H=H/H^-\).
\end{enumerate}
Let
\[
 \mathsf{MSub}(G):=
 \{K\leq G:K\text{ is a module of the underlying ordered set}\},
\]
ordered by inclusion.  Consequently, for every \(H\in\cV(G)\),
\[
 [H^-,H]_{\mathsf{MSub}(G)}\cong
 \begin{cases}
  \mathbf 2,&Q_H\text{ prime},\\
  \operatorname{ConvSub}(Q_H),&Q_H\text{ totally ordered},\\
  \operatorname{Sub}(Q_H),&Q_H\text{ equality-ordered}.
 \end{cases}
\]
Moreover, \(\mathsf{MSub}(G)\) is a complete sublattice of
\(\operatorname{Sub}(G)\).
\end{theorem}

Thus all genuine branching among subgroup-modules is localized in
equality-ordered canonical factors.  Further consequences include the
localization of every prime ordered subgroup of size at least three inside a
single prime canonical factor, and the finiteness of \(\cV(G)\) whenever the
underlying order has no infinite antichain.

We next turn to \(N\)-free orders.  Let \(N\) be the four-element poset with
strict comparabilities \(a<b,\qquad c<b,\qquad c<d.\)
A poset is \emph{\(N\)-free} if it has no induced copy of \(N\).  Equivalently,
its comparability graph contains no induced \(P_4\), so it is a cograph
\cite{Lerchs71,Lerchs72}.  For finite posets this is the familiar
series--parallel class, obtained from one-element posets by repeated linear
and disjoint sums.  The infinite situation is subtler: a cograph may be both
connected and co-connected, although additional finiteness hypotheses again
force strong series--parallel structure \cite{Zaguia2024}.

For an \(N\)-free poset the prime Gallai alternative disappears.  The central
question is therefore which alternating chain/antichain decomposition
patterns are compatible with the group operation.  We use the
Hahn--Pouzet--Woodrow representation of cographs by densely
\(0/1\)-valued robust-module meet-trees \cite{HahnPouzetWoodrow2024}.  On
the identity branch, the two labels become two kinds of canonical quotient:
a totally ordered factor, which we denote by type \(\Ctype\), and an
equality-ordered factor, of type \(\Atype\).

The group operation imposes further structure.  The
identity branch consists of convex subgroups; for each level \(H\), the
subgroup \(H^-\) is normal in \(H\); conjugation acts equivariantly on the
chain and its ordered factors; and the order is determined by the first
nontrivial quotient layer of an element.  Conversely, these conditions are
sufficient and give exactly the data needed for the converse construction.

Let \(G\) be a group.  A \emph{two-coloured subgroup chain} on \(G\) is a
chain \(\cH\) of subgroups of \(G\), each member being assigned one of two
types, \(\Ctype\) or \(\Atype\).  For \(H\in\cH\), put \(H^-:=\{\e\}\cup\bigcup\{K\in\cH:K<H\}.\)
The four admissibility conditions below express, respectively, existence of
a leading layer for every nonidentity element, normality of the lower
subgroup, the prescribed total/equality quotient type, and conjugation
equivariance.  We call \(\cH\) \emph{admissible} if:
\begin{enumerate}[label=\textup{(\roman*)}]
\item every \(g\ne\e\) belongs to a least member \(H_g\) of \(\cH\);

\item \(H^-\) is a proper normal subgroup of \(H\) for every \(H\in\cH\);

\item if \(H\) has type \(\Ctype\), then \(H/H^-\) is equipped with a
total group order, while if \(H\) has type \(\Atype\), then \(H/H^-\)
has the equality order;

\item conjugation by every element of \(G\) preserves the chain and its
types, and induces order isomorphisms between the corresponding
\(\Ctype\)-quotients.
\end{enumerate}

The chain is \emph{reduced} if, whenever \(K<H\) have the same type, there
exists \(L\in\cH\) of the opposite type such that \(K<L<H.\)

\begin{theorem}[Representation theorem for \(N\)-free ordered groups]
\label{thm:intro-characterisation}
Let \(G\) be a group.  The \(N\)-free two-sided invariant partial orders
on \(G\) are in one-to-one correspondence with the reduced admissible
two-coloured subgroup chains on \(G\).

Given such a chain \(\cH\), the corresponding order is determined as
follows.  For \(g\ne\e\), let \(H_g\) be the least member of \(\cH\)
containing \(g\).  Then
\[
 g>\e
 \quad\Longleftrightarrow\quad
 H_g\text{ has type }\Ctype
 \text{ and }
 gH_g^->H_g^-
 \text{ in }H_g/H_g^-,
\]
while \(g\parallel\e\) when \(H_g\) has type \(\Atype\); for arbitrary
\(x,y\in G\),
\[
 x<y\quad\Longleftrightarrow\quad x^{-1}y>\e.
\]

Conversely, if \(G\) is an \(N\)-free ordered group, its canonical chain
\(\cV(G)\), with a level \(H\) coloured \(\Ctype\) when \(H/H^-\) is
totally ordered and \(\Atype\) when \(H/H^-\) is equality-ordered, is
the corresponding reduced admissible chain.  In particular,
\[
 H_g=S(\e,g)
 \qquad(g\ne\e).
\]
\end{theorem}

The representation theorem is sharp at the level of value chains: every
reduced two-coloured chain occurs as the canonical value chain of an
\(N\)-free ordered abelian group whose canonical factors are all isomorphic
as groups to \(\mathbb Z\).  Consequently every nonempty chain occurs as
\(\cV(G)\) for some \(N\)-free ordered abelian group.

The decomposition also restricts canonically to subgroups, but in the
infinite case new canonical levels may arise as unions of ambient levels.
The exact restriction rule is the following.

\begin{theorem}[Subgroup restriction]
\label{thm:intro-subgroup-restriction}
Let \(G\) be an \(N\)-free ordered group and \(K\leq G\).  Put
\[
 \mathcal A_K:=
 \{S_G(\e,k):k\in K\setminus\{\e\}\}\subseteq\cV(G).
\]
Let \(\mathcal B_K\) be the chain of maximal convex subsets of
\(\mathcal A_K\) on which the type is constant, and for
\(H\in\mathcal A_K\) let \([H]\) denote the member of \(\mathcal B_K\)
containing \(H\).  For \(B\in\mathcal B_K\), define
\begin{equation}
 M_B:=\{\e\}\cup
 \{k\in K\setminus\{\e\}:[S_G(\e,k)]\leq B\}.
 \label{eq:intro-restricted-level}
\end{equation}
Then
\[
 \cV(K)=\{M_B:B\in\mathcal B_K\},
\]
and \(M_B\) has the common type of the ambient levels in \(B\).  Thus the
canonical decomposition of \(K\) is the reduced restriction of that of
\(G\).  When \(K\) meets only finitely many ambient canonical levels, this
simply deletes ambient levels not met by \(K\) and coalesces consecutive
remaining levels of the same type.
\end{theorem}

The finite-width consequences become especially explicit.  An \(N\)-free
ordered group has finite width if and only if its canonical chain is finite
and each equality-ordered canonical factor is finite.  If \(\{\e\}=H_0<H_1<\cdots<H_r=G\)
is the finite canonical chain, then
\[
 \width(G)=
 \prod_{H_i/H_{i-1}\text{ of type }\Atype}|H_i/H_{i-1}|.
\]
For every subgroup \(K\leq G\), \(\width(K)\mid\width(G),\)
and, more generally, if \(L\triangleleft K\) is a subgroup-module of the
induced order on \(K\), then \(\width(K/L)\mid\width(K)\mid\width(G).\)
In this finite-width setting the least canonical level can moreover be
recognized intrinsically from the comparability and incomparability
structure around \(\e\).  Successively recognizing and quotienting by
the least level reconstructs the canonical chain from the bottom up.

The \(N\)-free class lies inside dimension two even without a finiteness
hypothesis.  Indeed, every finite subposet of an \(N\)-free poset is again
\(N\)-free and hence series--parallel, so it has dimension at most two (the
property is preserved by linear and disjoint sums).  By the compactness
property of finite Dushnik--Miller dimension \cite{KellyTrotter1982}, the
whole poset therefore has dimension at most two.  Thus the present paper
describes the \(N\)-free part of the broader class of ordered groups of
dimension at most two.  A natural next step is to understand canonical
prime Gallai factors that themselves have dimension at most two, and the
additional constraints imposed on such factors by the group operation.

Section~\ref{sec:prelim} recalls the relational decomposition machinery,
and Section~\ref{sec:gallai-groups} proves the compatible-relation subgroup
lemma and the general ordered-group Gallai decomposition.
Sections~\ref{sec:subgroup-modules}--\ref{sec:canonical-consequences}
treat subgroup-modules, overlap and inheritance, quotients, and general
finiteness and width consequences.  Section~\ref{sec:nfree} proves the
\(N\)-free representation and subgroup-restriction theorems,
Section~\ref{sec:finite-width} treats finite width, and the final section
proves the realisation theorem and gives examples.

\section{Modules, strong modules, and Gallai decomposition}
\label{sec:prelim}

The modular decomposition is naturally formulated for arbitrary binary
relations.  We recall lexicographical sums, strong and robust modules, and
Gallai quotients, then specialize to posets and record the order-theoretic
facts needed later.  Finally, we recall the robust-module tree of a cograph
and its interpretation for \(N\)-free posets.

\subsection{Binary relations and Gallai decomposition}

Let \(\sigma\) be a binary relation on a set \(I\), and for each \(i\in I\)
let \(\rho_i\) be a binary relation on a nonempty set \(X_i\), where the
sets \(X_i\) are pairwise disjoint.  The \emph{lexicographical sum} of the
structures \((X_i,\rho_i)\), indexed by \((I,\sigma)\), is the binary
relation \(\rho\) on \(X:=\bigcup_{i\in I}X_i\)
defined, for \(x\in X_i\) and \(y\in X_j\), by
\[
 x\,\rho\,y
 \quad\Longleftrightarrow\quad
 \bigl(i=j\text{ and }x\,\rho_i\,y\bigr)
 \ \text{ or }\
 \bigl(i\ne j\text{ and }i\,\sigma\,j\bigr).
\]
The structures \((X_i,\rho_i)\) are called the \emph{summands}, and
\((I,\sigma)\) is the \emph{index structure}.

\begin{definition}\label{def:module}
Let \(\rho\) be a binary relation on a set \(X\).  A subset \(M\subseteq X\)
is a \emph{module} of \((X,\rho)\) if, for every \(x\in X\setminus M\) and
every \(m,m'\in M\),
\[
 x\,\rho\,m\quad\Longleftrightarrow\quad x\,\rho\,m',
 \qquad
 m\,\rho\,x\quad\Longleftrightarrow\quad m'\,\rho\,x.
\]
The empty set, the singletons of \(X\), and \(X\) itself are modules,
called \emph{trivial}.  The structure \((X,\rho)\) is \emph{prime} if it
has no nontrivial module.
\end{definition}

Modules are also called \emph{intervals} or \emph{autonomous sets}.  The
notion goes back to Fra\"{\i}ss\'e
\cite{fraisse53,fraisse84} and Gallai \cite{Gallai1967}; see also
\cite{ille-woodrow-2009}.

We use the following standard closure properties; see
\cite[Lemma~4.1]{CourcelleDelhomme2008}.

\begin{lemma}\label{lem:module-closure}
Let \((X,\rho)\) be a binary relation.
\begin{enumerate}[label=\textup{(\arabic*)}]
\item The intersection of a nonempty family of modules is a module
      (possibly empty).
\item The union of two modules with nonempty intersection is a module.
\item If \(M\) and \(M'\) are modules and
      \(M\setminus M'\ne\varnothing\), then
      \(M'\setminus M\) is a module.
\end{enumerate}
\end{lemma}

A nonempty module \(M\) is \emph{strong} if every module is either
disjoint from \(M\) or comparable with \(M\) under inclusion.  The
singletons and \(X\) are trivial strong modules.  A nonempty intersection
of strong modules is strong, and the union of a nonempty directed family
of strong modules is strong; see
\cite[Lemma~4.2]{CourcelleDelhomme2008}.

Following Courcelle and Delhomm\'e \cite{CourcelleDelhomme2008}, for a
binary relation \((X,\rho)\) and every nonempty \(A\subseteq X\), let
\(S_\rho(A)\) denote the intersection of all strong modules containing
\(A\).  Thus \(S_\rho(A)\) is the least strong module containing \(A\).
For \(x,y\in X\), put \(S_\rho(x,y):=S_\rho(\{x,y\}).\)
A strong module is called \emph{robust} if it is of the form
\(S_\rho(x,y)\) for some \(x,y\in X\), equivalently if it is \(S_\rho(F)\)
for some nonempty finite set \(F\); see
\cite[Lemma~3.1(4)]{CourcelleDelhomme2008}.  When the relation is clear,
we omit the subscript \(\rho\).

Let \(M\) be a strong module of \((X,\rho)\).  For \(x,y\in M\), define \(x\equiv_M y\)
if either \(x=y\), or there exists a strong module containing \(x\) and
\(y\) and properly contained in \(M\).

We use \cite[Lemma~6.5]{HahnPouzetWoodrow2024}.

\begin{lemma}\label{lem:gallai-blocks}
Let \((X,\rho)\) be a binary relation and let \(M\) be a strong module.
Then \(\equiv_M\) is an equivalence relation on \(M\), and its classes
are strong modules.  If \(|M|>1\), then \(\equiv_M\) has at least two
classes if and only if \(M\) is robust.  In that case its classes are
precisely the maximal proper strong submodules of \(M\).
\end{lemma}

When \(M\) is robust and has at least two elements, the equivalence
classes of \(\equiv_M\) are called the \emph{Gallai blocks} of \(M\).
They need not themselves be robust.

If \(B\) and \(C\) are distinct Gallai blocks, the relation between them
is uniform: for \(b,b'\in B\) and \(c,c'\in C\),
\[
 b\,\rho\,c\quad\Longleftrightarrow\quad b'\,\rho\,c',
 \qquad
 c\,\rho\,b\quad\Longleftrightarrow\quad c'\,\rho\,b'.
\]
Hence \(\rho\) induces a binary relation on the set of Gallai blocks.
This is the \emph{Gallai quotient} of \(M\), and
\((M,\rho_{\restriction M})\) is the lexicographical sum of its Gallai
blocks indexed by this quotient.  The Gallai quotient has no nontrivial
strong module.

For a binary relation, we say that it is \emph{edge-free} if no two
distinct elements are related in either direction, and \emph{complete}
if every two distinct elements are related in both directions.  The general
Gallai trichotomy follows.  The finite theory goes
back to Gallai \cite{Gallai1967}; for extensions to general binary
relations and to the infinite setting, see
\cite{EhrenfeuchtRozenbergII,HarjuRozenberg1994,
CourcelleDelhomme2008}.

\begin{theorem}\label{thm:no-strong}
A binary relation \((X,\rho)\) with at least two elements has no
nontrivial strong modules if and only if it is prime, edge-free,
complete, or a linear order.
\end{theorem}

We also use the following general alternation property,
Proposition~6.10 of \cite{HahnPouzetWoodrow2024}.

\begin{proposition}[Hahn--Pouzet--Woodrow]
\label{prop:HPW-alternation}
Let \(B\subset A\) be nonsingleton robust modules of a binary relation.
Suppose that the Gallai quotients of \(A\) and \(B\) are of the same
nonprime kind: both edge-free, both complete, or both linear orders.
Then there is a robust module \(C\) with
\[
 B\subset C\subset A
\]
whose Gallai quotient is not of that same kind.
\end{proposition}

\subsection{Specialization to posets}
\label{sec:gallai-posets}

We specialize the preceding decomposition to partial orders.  When
the index relation and all the summand relations are partial orders, the
lexicographical sum defined above is the usual lexicographical sum of
posets.  If the index poset is a chain, the sum is called a
\emph{linear sum}; if the index poset is an antichain, it is called a
\emph{disjoint sum}.

Let \(P=(X,\leq)\) be a poset.  In terms of the strict order, a subset
\(M\subseteq X\) is a module precisely when, for every
\(x\in X\setminus M\) and every \(m,m'\in M\),
\[
 x<m\;\Longrightarrow\;x<m',
 \qquad
 m<x\;\Longrightarrow\;m'<x.
\]
The reverse implications follow by interchanging \(m\) and \(m'\).

For posets there is an additional property that will be used repeatedly.

\begin{lemma}\label{lem:1}
Every module of a poset is convex.
\end{lemma}

For example, the modules of a chain are its ordinary intervals, so a
chain with at least three elements is not prime.

The \emph{comparability graph} of a poset \(P=(X,\leq)\), denoted
\(\Comp(P)\), is the graph with vertex set \(X\) in which two distinct
vertices \(x,y\) are adjacent if and only if they are comparable in \(P\);
that is,
\[
 xy\in E(\Comp(P))
 \quad\Longleftrightarrow\quad
 x<y\ \text{ or }\ y<x.
\]
For \(x\in X\), write \(\inc_P(x):=\{y\in X:y\parallel x\},\)
and let \(\Inc(P)\) denote the \emph{incomparability graph} of \(P\), in
which two distinct vertices are adjacent exactly when they are
incomparable.  We omit the subscript from \(\inc_P(x)\) when the ambient
poset is clear.

Gallai proved in the finite case that a poset and its comparability graph
have the same strong modules \cite{Gallai1967}, and Kelly
\cite{Kelly1985} extended the relevant decomposition results to the
infinite setting.  Hence they have the same robust modules and the same
least strong modules \(S(x,y)\).  We also use the classical
equivalence
\[
 P\text{ is prime}
 \quad\Longleftrightarrow\quad
 \Comp(P)\text{ is prime}.
\]

For posets, \cref{thm:no-strong} simplifies considerably.  The edge-free
case is an antichain, the complete case is excluded by antisymmetry, and
the linear-order case is a chain.  Hence a poset with at least two
elements has no nontrivial strong modules if and only if it is prime, a
chain, or an antichain.

Consequently, if \(M\) is a nonsingleton robust module of a poset \(P\),
then \(P_{\restriction M}\) is the lexicographical sum of its Gallai
blocks indexed by a quotient which is either prime, a chain, or an
antichain.

\subsection{The robust-module tree of a cograph}
\label{sec:cograph-tree}

We recall only the features of the robust-module decomposition of
cographs needed below; see Hahn, Pouzet and Woodrow \cite{HahnPouzetWoodrow2024}.

For a graph \(\Gamma\), let \(\mathcal R(\Gamma)\) denote its robust
modules, including the singletons, ordered by reverse inclusion.  If
\(A\in\mathcal R(\Gamma)\) is nonsingleton and \(\Gamma\) is a cograph,
the Gallai quotient of \(A\) is either a complete graph or an independent
set.  Put
\[
 \tau_\Gamma(A)=
 \begin{cases}
  1,&\text{if the quotient is complete},\\
  0,&\text{if the quotient is independent}.
 \end{cases}
\]
We use the term \emph{ramified meet-tree} in the sense of
Hahn--Pouzet--Woodrow \cite{HahnPouzetWoodrow2024}.  A
\(\{0,1\}\)-valuation on such a tree is called \emph{dense} if, whenever
\(a<b\) are nonsingleton nodes, there is a node \(c\) with
\[
 a<c\leq b
 \qquad\text{and}\qquad
 \tau(c)\ne\tau(a).
\]
In particular, if the two endpoints have the same value, an intermediate
node has the opposite value.

The form of the cograph decomposition used here combines Lemma~6.15,
Lemma~6.18 and Theorem~6.19 of Hahn, Pouzet and
Woodrow \cite{HahnPouzetWoodrow2024}.  Although the main results of that
paper concern countable cographs, these decomposition statements in its
appendix are formulated for arbitrary cographs.

\begin{theorem}[Hahn--Pouzet--Woodrow]\label{thm:HPW-tree}
Let \(\Gamma\) be a cograph.  Then
\((\mathcal R(\Gamma),\tau_\Gamma)\) is a ramified meet-tree with dense
\(\{0,1\}\)-valuation.  If \(x\ne y\) are vertices, then
\[
 xy\in E(\Gamma)
 \quad\Longleftrightarrow\quad
 \tau_\Gamma(S_\Gamma(x,y))=1.
\]
Conversely, if \((T,\tau)\) is a densely \(\{0,1\}\)-valued ramified
meet-tree and its maximal elements are taken as vertices, joining
\(x\ne y\) exactly when \(\tau(x\wedge y)=1\), then the resulting graph
is a cograph and its valued robust-module tree is isomorphic to
\((T,\tau)\).
\end{theorem}

By the Gallai--Kelly correspondence above, for posets this tree may be
read directly from the comparability graph.  Thus, if \(P\) is
\(N\)-free, then \(\Comp(P)\) is a cograph and the
Hahn--Pouzet--Woodrow tree may be read as the robust-module tree of
\(P\).  At a nonsingleton robust module \(A\), label \(0\) means that
the Gallai quotient of \(P[A]\) is an antichain, while label \(1\) means
that it is a chain: an independent quotient has no comparable pair,
whereas a complete quotient is totally ordered by the transitive
orientation.

The \(0/1\)-tree determines the comparability graph.  To recover the
poset itself, one must in addition retain the orientation of each
complete quotient; equivalently, one retains the total order of its
Gallai blocks.

\section{Compatible relations on groups and the identity branch}
\label{sec:gallai-groups}

We prove \cref{thm:intro-gallai} without an \(N\)-free hypothesis, first
isolating the algebraic step in a form independent of order.

Let \(G\) be a group and let \(\rho\) be a binary relation on \(G\).  We say
that \(\rho\) is \emph{compatible with the group operation} if every left and
right translation preserves and reflects \(\rho\); equivalently, for all
\(a,b,x,y\in G\),
\begin{equation}
 x\,\rho\,y
 \quad\Longleftrightarrow\quad
 axb\,\rho\,ayb.
 \label{eq:compatible-relation}
\end{equation}
For a partial order, this is exactly two-sided invariance, so compatible
partial orders are precisely the orders considered in this paper.  Since
inverse translations are again translations, it is enough to require the
forward implication in \eqref{eq:compatible-relation}.

Compatibility also implies
\begin{equation}
 x\,\rho\,y
 \quad\Longleftrightarrow\quad
 y^{-1}\,\rho\,x^{-1},
 \label{eq:inverse-anti}
\end{equation}
because one may multiply on the left by \(y^{-1}\) and on the right by
\(x^{-1}\).  Thus inversion is a relation anti-automorphism.  In particular,
left and right translations, as well as inversion, carry modules to modules
and strong modules to strong modules.  Hence
\begin{align}
 aS_\rho(x,y)&=S_\rho(ax,ay),&
 S_\rho(x,y)a&=S_\rho(xa,ya),\label{eq:translate-S}\\
 S_\rho(x,y)^{-1}&=S_\rho(x^{-1},y^{-1}).
 \label{eq:inverse-S}
\end{align}

\begin{theorem}[Robust subgroup theorem]
\label{thm:robust-subgroup}
Let \(G\) be a group equipped with a binary relation \(\rho\) compatible with
the group operation, and let \(x\ne\e\).  Then
\[
 S_\rho(\e,x)
\]
is a subgroup of \(G\).  If \(\rho\) is a partial order, this subgroup is
convex.
\end{theorem}

\begin{proof}
Put \(M:=S_\rho(\e,x).\)
By \eqref{eq:inverse-S}, \(M^{-1}\) is strong and contains \(\e\).  Since
\(M\) and \(M^{-1}\) meet, they are nested.  If, for instance,
\(M\subseteq M^{-1}\), then taking inverses gives \(M^{-1}\subseteq M.\)
The other case is symmetric.  Hence
\begin{equation}
 M^{-1}=M.
 \label{eq:M-inverse}
\end{equation}

If \(a,b\in M\), then
\begin{equation}
 S_\rho(a,b)\subseteq M,
 \label{eq:S-inside-M}
\end{equation}
because \(M\) is strong and contains \(a\) and \(b\).

We first show that \(xa^{-1}\in M \qquad(a\in M).\)
By \eqref{eq:translate-S}, \eqref{eq:inverse-S}, and
\eqref{eq:S-inside-M},
\[
 S_\rho(\e,xa^{-1})
   =xS_\rho(x,a)^{-1}
   \subseteq xM^{-1}.
\]
Moreover, by \eqref{eq:inverse-S} and \eqref{eq:translate-S},
\[
 xM^{-1}
   =xS_\rho(\e,x^{-1})
   =S_\rho(x,\e)
   =M.
\]
Thus \(xa^{-1}\in M\).

Now let \(a,b\in M\).  Right translation gives
\[
 S_\rho(\e,ba^{-1})
   =S_\rho(a,b)a^{-1}
   \subseteq Ma^{-1}
   =S_\rho(a^{-1},xa^{-1}).
\]
By \eqref{eq:M-inverse}, \(a^{-1}\in M\), and we have just shown that
\(xa^{-1}\in M\).  Hence \eqref{eq:S-inside-M} gives \(S_\rho(a^{-1},xa^{-1})\subseteq M.\)
Therefore \(ba^{-1}\in M\), so \(M\) is a subgroup.  If \(\rho\) is a
partial order, \(M\) is a module of the corresponding poset and is convex by
\cref{lem:1}.
\end{proof}

\begin{corollary}\label{cor:translate-robust}
Let \(G\) be a group equipped with a compatible binary relation \(\rho\).
Every nonsingleton robust module is a left coset of a subgroup and also a
right coset of a subgroup.  More precisely,
\[
 S_\rho(x,y)
   =xS_\rho(\e,x^{-1}y)
   =S_\rho(\e,yx^{-1})x.
\]
If \(\rho\) is a partial order, the two subgroups appearing here are convex.
\end{corollary}

\begin{proof}
Since \(x\ne y\), both \(x^{-1}y\) and \(yx^{-1}\) are nonidentity
elements.  The two identities follow from \eqref{eq:translate-S}, and the
two identity modules are subgroups by \cref{thm:robust-subgroup}.  In the
ordered case they are convex by the same theorem.
\end{proof}

We specialize to the order relation of an ordered group and suppress the
subscript \(\rho\) from \(S_\rho\).

We use the following standard quotient construction.  Let \(G\) be a
group, let \(H\) be a normal subgroup of \(G\), and suppose that \(G\) is
equipped with a compatible order relation \(\leq\).  For cosets
\(\alpha,\beta\in G/H\), define \(\alpha\leq_H\beta\)
if \(a\leq b\) for some \(a\in\alpha\) and \(b\in\beta\).
Equivalently, for every \(a\in\alpha\) there exists \(b\in\beta\) such that
\(a\leq b\).  The relation \(\leq_H\) is a compatible quasi-order on
\(G/H\); see
\cite[Proposition~4, p.~25]{KopytovMedvedev}.

\begin{lemma}\label{convex}
Let \(G\) be an ordered group and let \(H\) be a normal subgroup of
\(G\).  Then \(\leq_H\) is a compatible order relation on \(G/H\) if
and only if \(H\) is convex in \(G\).

Moreover, if \(H\) is a module of the ordered set \((G,\leq)\), then
the order on \(G\) is the lexicographical sum of its cosets of \(H\),
each isomorphic to \(H\), indexed by the ordered quotient
\((G/H,\leq_H)\).
\end{lemma}

\begin{proof}
The first assertion is \cite[Proposition~4, p.~25]{KopytovMedvedev}.
For the second, suppose that \(H\) is a module.  By \cref{lem:1},
\(H\) is convex, so \((G/H,\leq_H)\) is an ordered group.  Every coset
of \(H\) is a module and is order-isomorphic to \(H\), since
translations are order automorphisms.  If distinct cosets
\(\alpha,\beta\) satisfy \(\alpha<_H\beta\), choose
\(a\in\alpha\) and \(b\in\beta\) with \(a<b\).  The module property
of \(\alpha\) and \(\beta\) then gives \(a'<b'\qquad(a'\in\alpha,\ b'\in\beta).\)
The reverse case is symmetric, while incomparable quotient elements
have no comparable representatives.  Hence the order on \(G\) is the
lexicographical sum of its cosets of \(H\), indexed by
\((G/H,\leq_H)\).
\end{proof}

We next transport the subgroup structure from
\cref{thm:robust-subgroup} through the Gallai decomposition.

\begin{lemma}[The robust identity branch]\label{lem:robust-identity-branch}
The nonsingleton robust modules of an ordered group \(G\) that contain
\(\e\) are exactly the members of \(\cV(G)\).
\end{lemma}

\begin{proof}
Every member of \(\cV(G)\) is robust and contains \(\e\) by definition.
Conversely, let \(A\) be a nonsingleton robust module containing \(\e\).
By \cref{lem:gallai-blocks}, \(A\) has at least two maximal proper
strong submodules.  Let \(B\) be the one containing \(\e\), and choose
\(g\in A\setminus B\).  The strong module \(S(\e,g)\) is contained in
\(A\), meets \(B\), and is not contained in \(B\).  Hence \(B<S(\e,g)\leq A.\)
Maximality of \(B\) gives \(S(\e,g)=A\).  Thus \(A\in\cV(G)\).
\end{proof}

For \(g\ne\e\), write \(v(g):=S(\e,g).\)

\begin{proof}[Proof of \cref{thm:intro-gallai}]
By \cref{thm:robust-subgroup}, every member of \(\cV(G)\) is a convex
subgroup of \(G\).  Any two members of \(\cV(G)\) contain \(\e\), so,
being strong modules, they are comparable under inclusion.  Hence
\(\cV(G)\) is a chain.

We first describe all strong modules containing \(\e\).  Let
\(I\subseteq\cV(G)\) be an initial segment and put \(H_I:=\{\e\}\cup\bigcup_{H\in I}H.\)
If \(I=\varnothing\), then \(H_I=\{\e\}\).  If \(I\ne\varnothing\),
then \(H_I\) is the union of a chain of subgroups and is therefore a
subgroup.  It is strong because a nonempty directed union of strong
modules is strong.  Hence \(H_I\) is convex by \cref{lem:1}.

Conversely, let \(M\) be a strong module containing \(\e\).  If
\(M=\{\e\}\), then \(M=H_\varnothing\).  Suppose
\(M\ne\{\e\}\).  For every \(x\in M\setminus\{\e\}\), minimality of
\(S(\e,x)\) gives \(S(\e,x)\subseteq M.\)
Therefore \(M=\bigcup\{S(\e,x):x\in M\setminus\{\e\}\}.\)
Put
\[
 I_M:=\{S(\e,x):x\in M\setminus\{\e\}\}\subseteq\cV(G).
\]
If \(K<S(\e,x)\) with \(S(\e,x)\in I_M\) and \(K\in\cV(G)\), then \(K\subseteq S(\e,x)\subseteq M.\)
Writing \(K=S(\e,y)\), we have \(y\in M\), and therefore
\(K\in I_M\).  Thus \(I_M\) is an initial segment and \(M=H_{I_M}.\)

The initial segment is unique.  Indeed, for every initial segment
\(I\subseteq\cV(G)\), \(I=\{S(\e,g):g\in H_I\setminus\{\e\}\}.\)
For if \(H=S(\e,g)\in I\), then \(g\in H\subseteq H_I\).  Conversely,
if \(g\in H_I\setminus\{\e\}\), then \(g\in K\) for some \(K\in I\).
Since \(K\) is strong and contains \(\e\) and \(g\), minimality gives
\(S(\e,g)\subseteq K\), and the initial-segment property yields
\(S(\e,g)\in I\).

Finally, let \(M\) be any strong module of \(G\) and choose \(a\in M\).
Left translation by \(a^{-1}\) is an order automorphism, so
\(a^{-1}M\) is a strong module containing \(\e\).  Hence \(a^{-1}M=H_I\)
for a unique initial segment \(I\subseteq\cV(G)\), and therefore \(M=aH_I.\)
This proves \textup{(a)}.

Fix now \(H\in\cV(G)\), say \(H=S(\e,x)\).  Since \(H\) is a strong
module, the strong modules of the induced order on \(H\) are precisely
the strong modules of \(G\) contained in \(H\).  Since \(H^-=\{\e\}\cup\bigcup\{K\in\cV(G):K<H\},\)
part \textup{(a)} shows that \(H^-\) is a strong convex subgroup of
\(H\).  It is proper: if \(x\in H^-\), then \(x\in K\) for some
\(K\in\cV(G)\) with \(K<H\), and minimality of \(S(\e,x)\) would give \(H=S(\e,x)\subseteq K,\)
a contradiction.

We claim that \(H^-\) is the maximal proper strong submodule of \(H\)
containing \(\e\).  Let \(M\) be such a submodule.  For every
\(y\in M\setminus\{\e\}\), minimality gives \(S(\e,y)\subseteq M<H.\)
Thus \(S(\e,y)<H\), and hence \(y\in H^-\).  Therefore
\(M\subseteq H^-\), proving the claim.

For \(h\in H\), conjugation by \(h\) is an order automorphism fixing
\(\e\) and carrying \(H\) to itself.  It therefore fixes the unique
maximal proper strong submodule of \(H\) containing \(\e\).  Hence \(hH^-h^{-1}=H^-,\)
so \(H^-\triangleleft H\).

Since \(H\) is robust, its maximal proper strong submodules are its
Gallai blocks.  The block containing \(\e\) is \(H^-\).  For each
\(h\in H\), left translation by \(h\) is an order automorphism of the
induced order on \(H\), and hence the Gallai block containing \(h\) is \(hH^-.\)
Thus the Gallai blocks of \(H\) are precisely the left cosets of
\(H^-\).

Because \(H^-\) is a proper convex normal subgroup of \(H\),
\cref{convex} shows that \(Q_H:=H/H^-\)
with its natural quotient order is an ordered group with at least two
elements.  Since the \(H^-\)-cosets are precisely the Gallai blocks of
\(H\), this natural quotient order is exactly the Gallai quotient order.
Hence \(Q_H\) has no nontrivial strong module.  The poset specialization
of \cref{thm:no-strong} therefore shows that \(Q_H\) is prime, a chain,
or an antichain.  Moreover, the induced order on \(H\) is the
lexicographical sum of its \(H^-\)-cosets indexed by \(Q_H\).  For an
ordered group, an antichain order is the equality order.  This proves
\textup{(b)}.

For \textup{(c)}, let \(g\ne\e\) and put \(H=S(\e,g)\).  As above,
\(g\notin H^-\), so \(H^-\) and \(gH^-\) are distinct Gallai blocks
of \(H\).  Their relation is uniform, and \(\e\in H^-, \qquad g\in gH^-.\)
Consequently the image \(gH^-\) in \(Q_H\) determines the comparison
of \(g\) with \(\e\).  In particular,
\begin{equation}
 g>\e
 \quad\Longleftrightarrow\quad
 gH^->H^-
 \text{ in }Q_H.
 \label{eq:general-leading}
\end{equation}
Since
\[
 x\leq y
 \quad\Longleftrightarrow\quad
 \e\leq x^{-1}y,
\]
this determines the partial order on \(G\) from the canonical subgroup
chain and its ordered Gallai quotients.

Finally, let \(a\in G\) and write \(H=S(\e,g)\).  Conjugation by \(a\)
is an order automorphism fixing \(\e\), and
\[
 aHa^{-1}
 =aS(\e,g)a^{-1}
 =S(\e,aga^{-1})
 \in\cV(G).
\]
It carries the Gallai block \(H^-\) containing \(\e\) onto the Gallai
block \((aHa^{-1})^-\) containing \(\e\).  Thus \(aH^-a^{-1}=(aHa^{-1})^-.\)
It follows that conjugation induces an isomorphism
\[
 H/H^-\longrightarrow aHa^{-1}/(aHa^{-1})^-,
 \qquad
 hH^-\longmapsto aha^{-1}(aHa^{-1})^-,
\]
of ordered groups.  This proves the equivariance assertion and completes
the proof.
\end{proof}

\begin{corollary}[Quotient compatibility]
\label{cor:quotient-compatibility}
Let \(G\) be an ordered group and let \(K\triangleleft G\) be a proper
strong module containing \(\e\).  Then
\[
 \cV(G/K)=\{H/K:H\in\cV(G),\ K<H\},
\]
and, for every \(g\notin K\),
\[
 S_{G/K}(K,gK)=S_G(\e,g)/K.
\]
Moreover, if \(H\in\cV(G)\) and \(K<H\), then \(K\leq H^-\),
\[
 (H/K)^-=H^-/K,
\]
and the corresponding canonical factor is naturally order-isomorphic to
\(H/H^-\).
\end{corollary}

\begin{proof}[Proof of Corollary \cref{cor:quotient-compatibility}]
By \cref{thm:intro-gallai}, \(K\) is convex, so the quotient order makes
\(G/K\) an ordered group.  By \cref{convex}, \(G\) is the
lexicographical sum of its \(K\)-cosets indexed by \(G/K\).  Hence
strong modules of \(G/K\) correspond, under inverse image, to strong
modules of \(G\) containing \(K\).

Let \(g\notin K\) and put \(H=S_G(\e,g)\).  Since \(K\) and \(H\) are
strong modules containing \(\e\), they are comparable.  As \(g\notin K\),
we have \(K<H\).  Hence \(S_{G/K}(K,gK)=H/K=S_G(\e,g)/K.\)
The description of \(\cV(G/K)\) follows.

If \(K<H\), then \(K\) is a proper strong submodule of \(H\) containing
\(\e\), and therefore \(K\leq H^-\).  Using the description of
\(\cV(G/K)\) above, we obtain \((H/K)^-=H^-/K.\)
Consequently \((H/K)/(H^-/K)\cong H/H^-,\)
naturally as ordered groups.
\end{proof}

\section{Subgroup-modules}
\label{sec:subgroup-modules}

\begin{proof}[Proof of \cref{thm:intro-subgroup-modules}]
Suppose first that \(K\) is a module of \(G\).  If \(K\) is strong,
then \cref{thm:intro-gallai}\textup{(a)} gives a unique initial segment
\(I\subseteq\cV(G)\) with \(K=H_I\), so \textup{(a)} holds.

Assume that \(K\) is not strong.  We claim that there is
\(x\in K\setminus\{\e\}\) such that \(S(\e,x)\nsubseteq K.\)
Indeed, otherwise every \(S(\e,x)\), \(x\in K\setminus\{\e\}\),
would be contained in \(K\), and
\[
 K=\{\e\}\cup\bigcup_{x\in K\setminus\{\e\}}S(\e,x).
\]
The modules in this union form a chain of strong modules, so their union
is strong, a contradiction.

Choose such an \(x\) and put \(H:=S(\e,x).\)
Since \(H\) is strong, \(K\) is a module, and \(H\cap K\) contains
\(\e\) and \(x\), the sets \(H\) and \(K\) are comparable under
inclusion.  As \(H\nsubseteq K\), it follows that \(K<H.\)

Let \(L\in\cV(G)\) with \(L<H\).  Since \(L\) is strong and
\(L\cap K\ne\varnothing\), the sets \(L\) and \(K\) are comparable.
We cannot have \(K\subseteq L\), since then \(x\in L\), whereas the
strong module \(L\) containing \(\e\) and \(x\) would force \(H=S(\e,x)\subseteq L,\)
contrary to \(L<H\).  Hence \(L\subseteq K\).  Therefore
\[
 H^-=\{\e\}\cup\bigcup_{\substack{L\in\cV(G)\\L<H}}L
 \subseteq K.
\]
Moreover \(x\notin H^-\): otherwise \(x\in L\) for some
\(L\in\cV(G)\) with \(L<H\), and the minimality of
\(H=S(\e,x)\) would give \(H\subseteq L\), a contradiction.  Since
\(x\in K\), the inclusion \(H^-\subseteq K\) is therefore strict.  Thus \(H^-<K<H.\)

The following local equivalence also proves the converse.
Whenever \(H^-\leq K\leq H,\)
we have
\begin{equation}
 K\text{ is a module of }G
 \quad\Longleftrightarrow\quad
 K/H^-\text{ is a module of }Q_H.
 \label{eq:subgroup-module-local}
\end{equation}

Since \(H^-\triangleleft H\), every subgroup \(K\) with
\(H^-\leq K\leq H\) is a union of \(H^-\)-cosets.  Suppose first that
\(K\) is a module of \(G\), and let \(zH^-\notin K/H^-.\)
Then \(z\notin K\).  For all \(k,k'\in K\),
\[
 z\leq k\quad\Longleftrightarrow\quad z\leq k',
 \qquad
 k\leq z\quad\Longleftrightarrow\quad k'\leq z.
\]
Since the natural order on \(Q_H\) is the Gallai quotient order, the same
uniformity holds between \(zH^-\) and the elements of \(K/H^-\).
Thus \(K/H^-\) is a module of \(Q_H\).

Conversely, suppose that \(K/H^-\) is a module of \(Q_H\), and let
\(z\in G\setminus K\).  If \(z\notin H\), then the module property of
\(H\) makes both relations between \(z\) and the elements of
\(K\subseteq H\) uniform.  If \(z\in H\setminus K\), then \(zH^-\notin K/H^-.\)
The module property of \(K/H^-\) in \(Q_H\), together with the
lexicographical decomposition of \(H\) into its \(H^-\)-cosets, again
makes both relations between \(z\) and the elements of \(K\) uniform.
Hence \(K\) is a module of \(G\).  This proves
\eqref{eq:subgroup-module-local}.

The level \(H\) is unique.  Indeed, suppose also that \(H'^-<K<H'\)
for some \(H'\in\cV(G)\).  Since \(\cV(G)\) is a chain, \(H\) and
\(H'\) are comparable.  If, say, \(H<H'\), then by the definition of
\(H'^-\), \(H\subseteq H'^-.\)
Hence \(K<H\subseteq H'^-<K,\)
a contradiction.  The case \(H'<H\) is symmetric.

This proves the necessity in \textup{(b)}, while
\eqref{eq:subgroup-module-local} proves its sufficiency.  The
sufficiency of \textup{(a)} is already contained in
\cref{thm:intro-gallai}\textup{(a)}.

It remains to identify the subgroup-modules inside a canonical interval.
For each \(H\in\cV(G)\), the map \(K\longmapsto K/H^-\)
identifies \([H^-,H]_{\mathsf{MSub}(G)}\) with the lattice of
subgroup-modules of \(Q_H\).  If \(Q_H\) is prime, its only nonempty
modules are the singletons and the whole quotient, so its only
subgroup-modules are the trivial subgroup and \(Q_H\) itself.  If
\(Q_H\) is totally ordered, its modules are exactly its intervals; a
subgroup is therefore a module exactly when it is convex.  If \(Q_H\)
is equality-ordered, every subset is a module, so every subgroup is a
module.  This gives
\[
 [H^-,H]_{\mathsf{MSub}(G)}\cong
 \begin{cases}
  \mathbf 2,&Q_H\text{ prime},\\
  \operatorname{ConvSub}(Q_H),&Q_H\text{ totally ordered},\\
  \operatorname{Sub}(Q_H),&Q_H\text{ equality-ordered}.
 \end{cases}
\]

It remains to prove the complete-sublattice assertion.  The nullary meet
and join are \(G\) and \(\{\e\}\), respectively, and both are
subgroup-modules.  Let \(\{K_i:i\in I\}\) be a nonempty family of
subgroup-modules.  Its intersection is a subgroup and is again a module.
Indeed, if \(z\notin\bigcap_{i\in I}K_i,\)
choose \(i\in I\) with \(z\notin K_i\).  Since \(\bigcap_{i\in I}K_i\subseteq K_i,\)
the uniform relation of \(z\) to \(K_i\) restricts to a uniform relation
to the intersection.  Thus arbitrary meets in
\(\mathsf{MSub}(G)\) agree with those in \(\operatorname{Sub}(G)\).

For joins, first let \(A,B\in\mathsf{MSub}(G)\).  If they are
comparable, their subgroup-generated join is the larger one.  Suppose
that they are incomparable.  Since both contain \(\e\), they overlap.
Consequently neither can be strong.  Let \(H_A,H_B\) be their unique
canonical levels from \textup{(b)}.  If, say, \(H_A<H_B\), then \(A<H_A\leq H_B^-<B,\)
contradicting the incomparability of \(A\) and \(B\).  Hence \(H_A=H_B=:H.\)
The quotient subgroups \(A/H^-, \qquad B/H^-\)
are incomparable subgroup-modules of \(Q_H\).  This is impossible if
\(Q_H\) is prime, and it is also impossible if \(Q_H\) is totally
ordered, since the convex subgroups of a totally ordered group form a
chain.  Therefore \(Q_H\) is equality-ordered.

Every subgroup of an equality-ordered group is a module.  Hence \(\left\langle A/H^-,B/H^-\right\rangle\)
is a subgroup-module of \(Q_H\), and its inverse image under
\(H\to Q_H\) is a subgroup-module of \(G\).  Since \(H^-\leq A\cap B\),
that inverse image is precisely \(\langle A,B\rangle.\)
Thus binary joins in \(\mathsf{MSub}(G)\) agree with the corresponding
subgroup-generated joins.

Finite joins follow by induction.  For an arbitrary nonempty family
\(\{K_i:i\in I\}\), its subgroup-generated join is the union of the
directed family of its finite joins.  A directed union of
subgroup-modules is again a subgroup-module: if \(z\) lies outside the
union and \(x,y\) lie in it, some member of the directed family contains
both \(x\) and \(y\), and the module property there shows that \(z\) has
the same relation to both, in both directions.  Hence arbitrary joins
are preserved as well.

Therefore \(\mathsf{MSub}(G)\) is a complete sublattice of
\(\operatorname{Sub}(G)\).
\end{proof}

\begin{corollary}[Normal subgroup-modules]
\label{cor:normal-subgroup-modules}
Let \(G\) be an ordered group and let \(K\leq G\) be a subgroup-module.
Then the following describe exactly when \(K\triangleleft G\).
\begin{enumerate}[label=\textup{(\alph*)}]
\item If \(K\) is strong, write
\[
 K=H_I=\{\e\}\cup\bigcup_{H\in I}H
\]
for the unique initial segment \(I\subseteq\cV(G)\).  Then \(K\) is
normal if and only if \(I\) is invariant under the conjugation action of
\(G\) on \(\cV(G)\).

\item If \(K\) is not strong, let \(H\) be the unique canonical level with
\[
 H^-<K<H.
\]
Then \(K\) is normal if and only if \(H\triangleleft G\) and, under the
resulting conjugation action of \(G\) on
\[
 Q_H=H/H^-,
\]
the subgroup \(K/H^-\) is invariant.
\end{enumerate}
Consequently the normal subgroup-modules form a complete sublattice of
the normal-subgroup lattice of \(G\).
\end{corollary}

\begin{proof}
Suppose first that \(K=H_I\) is strong.  For \(a\in G\), conjugation
sends it to
\[
 aKa^{-1}=H_{aI},
 \qquad
 aI:=\{aHa^{-1}:H\in I\},
\]
by \cref{thm:intro-gallai}\textup{(c)}.  Thus \(K\) is fixed by every
inner automorphism if and only if \(I\) is invariant.

Suppose now that \(K\) is not strong and \(H^-<K<H.\)
If \(K\triangleleft G\), then for every \(a\in G\), conjugating this
strict containment and using the equivariance of the canonical
decomposition gives \((aHa^{-1})^-<K<aHa^{-1}.\)
The uniqueness of the local level in
\cref{thm:intro-subgroup-modules} therefore forces \(aHa^{-1}=H.\)
Hence \(H\triangleleft G\).  The equivariance assertion in
\cref{thm:intro-gallai}\textup{(c)} then also gives \(aH^-a^{-1}=H^-,\)
so conjugation by \(G\) induces an action by ordered-group
automorphisms on \(Q_H=H/H^-\).  Since \(K\triangleleft G\), the
subgroup \(K/H^-\) is invariant under this action.

Conversely, suppose that \(H\triangleleft G\) and that \(K/H^-\) is
invariant under the induced conjugation action of \(G\) on \(Q_H\).
Then \(K\), being the inverse image of \(K/H^-\) under \(H\longrightarrow Q_H,\)
is fixed by conjugation by every element of \(G\).  Hence
\(K\triangleleft G\).

Finally, arbitrary intersections and subgroup-generated joins of normal
subgroups are normal.  By \cref{thm:intro-subgroup-modules}, the same
operations preserve subgroup-modules.  Therefore the normal
subgroup-modules form a complete sublattice of the normal-subgroup
lattice of \(G\).
\end{proof}

\section{Subgroups crossing modular layers}
\label{sec:overlap-subgroups}

We next control how an arbitrary subgroup can meet a subgroup that is a
module.

\begin{theorem}[Overlap with a modular subgroup]
\label{thm:modular-subgroup-overlap}
Let \(G\) be an ordered group, let \(H\leq G\) be a module of the
ordered set \(G\), let \(K\leq G\), and put
\[
 L:=H\cap K.
\]
Then:
\begin{enumerate}[label=\textup{(\alph*)}]
\item \(L\) is a convex subgroup and a module of the induced order on
\(K\).

\item For every \(k\in K\),
\[
 K\cap kH=kL,
 \qquad
 K\cap Hk=Lk.
\]
Thus every left \(H\)-coset met by \(K\) is met in exactly one left
\(L\)-coset, and similarly on the right.

\item The map
\[
 kL\longmapsto kH
\]
is an order isomorphism from the set of left \(L\)-cosets of \(K\),
with the quotient order induced by their modular partition, onto the
subposet of the left \(H\)-coset quotient of \(G\) induced by the cosets
that meet \(K\).  Consequently the order on \(K\) is the
lexicographical sum of copies of \(L\) indexed by that induced subposet.

\item If, in addition, \(H\triangleleft G\), then
\(L\triangleleft K\), and the natural map
\[
 K/L\longrightarrow G/H,
 \qquad
 kL\longmapsto kH,
\]
is an embedding of ordered groups with image \(KH/H\).
\end{enumerate}
\end{theorem}

\begin{proof}
The intersection \(L=H\cap K\) is a subgroup.  If
\(z\in K\setminus L\), then \(z\notin H\).  Since \(H\) is a module
and \(L\subseteq H\), both relations between \(z\) and the elements of
\(L\) are uniform.  Hence \(L\) is a module of the induced order on
\(K\), and therefore it is convex by \cref{lem:1}.  This proves
\textup{(a)}.

For \(k\in K\), an element \(x\) belongs to \(K\cap kH\) if and only
if \(x=kh\)
for some \(h\in H\) and \(x\in K\).  Since \(k,x\in K\), this is
equivalent to \(h=k^{-1}x\in H\cap K=L.\)
Thus \(K\cap kH=kL.\)
The right-coset identity is proved in the same way.  This proves
\textup{(b)}.

By \textup{(a)} and translation invariance, the left cosets of \(L\)
are pairwise disjoint modules of the induced order on \(K\).  Likewise,
the left cosets of \(H\) are pairwise disjoint modules of \(G\).  Hence
both coset partitions carry their natural quotient orders.

The map \(kL\longmapsto kH\)
is well defined and injective because, for \(k_1,k_2\in K\),
\[
 k_1H=k_2H
 \quad\Longleftrightarrow\quad
 k_2^{-1}k_1\in H\cap K=L
 \quad\Longleftrightarrow\quad
 k_1L=k_2L.
\]
It is plainly onto the set of left \(H\)-cosets that meet \(K\).
Relations between distinct quotient blocks are the relations between
any chosen representatives.  Choosing representatives in \(K\)
therefore shows that this bijection preserves and reflects the order.
Since every left \(L\)-coset is order-isomorphic to \(L\), the order on
\(K\) is the lexicographical sum of these cosets indexed by the induced
subposet of the \(H\)-coset quotient.  This proves \textup{(c)}.

Finally, suppose \(H\triangleleft G\).  For every \(k\in K\),
\[
 kLk^{-1}
 =k(H\cap K)k^{-1}
 =kHk^{-1}\cap kKk^{-1}
 =H\cap K
 =L,
\]
so \(L\triangleleft K\).  The map
\[
 K/L\longrightarrow G/H,
 \qquad
 kL\longmapsto kH,
\]
is therefore the homomorphism induced by the inclusion \(K\hookrightarrow
G\).  Its kernel is \(K\cap H=L\), and its image is \(KH/H\).

Moreover, \(H\) is convex in \(G\) because it is a module, and \(L\)
is convex in \(K\) by \textup{(a)}.  Hence \cref{convex} identifies the
orders on \(G/H\) and \(K/L\) with the quotient orders of their
respective coset partitions.  The order-isomorphism established in
\textup{(c)} therefore shows that the displayed homomorphism is an
order embedding.  This proves \textup{(d)}.
\end{proof}

In particular, if \(H\) and \(K\) overlap and
\(H\cap K\ne\{\e\}\), then \(H\cap K\) is a nontrivial proper module
of \(K\); hence \(K\) is decomposable.  If instead
\(H\cap K=\{\e\}\), part \textup{(c)} identifies \(K\), as an ordered
set, with the subposet of the \(H\)-coset quotient induced by the
cosets met by \(K\).

The first consequence is a localization principle for prime ordered
subgroups.

\begin{corollary}[Prime-subgroup localization]
\label{cor:prime-localization}
Let \(G\) be an ordered group and let \(K\leq G\) be prime as an
ordered set, with \(|K|\geq3\).  Then there is a unique
\(H\in\cV(G)\) such that
\[
 K\leq H,
 \qquad
 K\cap H^-=\{\e\},
\]
and the natural map
\[
 K\longrightarrow Q_H=H/H^-,
 \qquad
 k\longmapsto kH^-,
\]
is an embedding of ordered groups.  Moreover \(Q_H\) is prime, and
\[
 S_G(\e,k)=H
\]
for every \(k\in K\setminus\{\e\}\).
\end{corollary}

\begin{proof}
Choose \(k\in K\setminus\{\e\}\) and put \(H:=S_G(\e,k).\)
By \cref{thm:intro-gallai}, \(H\) is a subgroup-module of \(G\).
Hence \cref{thm:modular-subgroup-overlap}\textup{(a)} shows that
\(K\cap H\) is a module of the induced order on \(K\).  It contains
the two distinct elements \(\e\) and \(k\), so primality of \(K\)
gives \(K\cap H=K.\)
Thus \(K\leq H.\)

Likewise, \(H^-\) is a subgroup-module of \(H\), and hence also a
module of \(G\).  Applying
\cref{thm:modular-subgroup-overlap}\textup{(a)} to \(H^-\) and \(K\)
shows that \(K\cap H^-\) is a module of \(K\).  Since
\(k\notin H^-\), this module is proper.  By primality, \(K\cap H^-=\{\e\}.\)

Now apply \cref{thm:modular-subgroup-overlap}\textup{(d)} inside the
ordered group \(H\), with the normal subgroup-module \(H^-\) and the
subgroup \(K\).  Since \(K\cap H^-=\{\e\},\)
the natural map
\[
 K\longrightarrow H/H^-,
 \qquad
 k\longmapsto kH^-,
\]
is an embedding of ordered groups.

Let \(x\in K\setminus\{\e\}\) and put \(H_x:=S_G(\e,x).\)
Repeating the argument above with \(x\) in place of \(k\) gives \(K\leq H_x.\)
Since \(H,H_x\in\cV(G)\), they are comparable.  If \(H_x<H\), then
\(k\in K\leq H_x\), while \(H=S_G(\e,k)\) is the least strong module
containing \(\e\) and \(k\); hence \(H\leq H_x\), a contradiction.
The case \(H<H_x\) similarly contradicts the minimality of
\(H_x=S_G(\e,x)\).  Therefore \(H_x=H.\)
Thus \(S_G(\e,x)=H \qquad (x\in K\setminus\{\e\}).\)

It remains to verify uniqueness of the canonical level.  Suppose
\(J\in\cV(G)\) also satisfies \(K\leq J, \qquad K\cap J^-=\{\e\}.\)
Choose \(k\in K\setminus\{\e\}\).  Since \(J\) is strong and contains
\(\e\) and \(k\), minimality gives \(H=S_G(\e,k)\leq J.\)
If \(H<J\), then by the definition of \(J^-\), \(H\subseteq J^-,\)
and therefore \(k\in K\cap J^-,\)
contrary to \(K\cap J^-=\{\e\}\).  Hence \(J=H\).

Finally, by \cref{thm:intro-gallai}\textup{(b)}, the ordered group
\(Q_H\) is prime, totally ordered, or equality-ordered.  If it were
totally ordered, its ordered subgroup \(K\) would be a chain; if it
were equality-ordered, \(K\) would be an antichain.  Since
\(|K|\geq3\), either possibility would give \(K\) a nontrivial module,
contrary to primality.  Hence \(Q_H\) is prime.
\end{proof}

Thus prime ordered structure cannot be assembled by crossing several
canonical Gallai levels: every prime ordered subgroup of size at least
three is inherited from a single prime canonical factor.

\begin{corollary}[Subgroup inheritance]
\label{cor:subgroup-inheritance}
Let \(G\) be an ordered group and let \(K\leq G\).  For
\(H\in\cV(G)\), put
\[
 K_H:=K\cap H,
 \qquad
 K_H^-:=K\cap H^-.
\]
Then \(\{K_H:H\in\cV(G)\}\) is a chain of convex subgroups of \(K\),
each of which is a module of the induced order on \(K\), and
\[
 K_H^-\triangleleft K_H.
\]
Moreover, the natural map
\[
 K_H/K_H^-\longrightarrow H/H^-,
 \qquad
 kK_H^-\longmapsto kH^-,
\]
is an embedding of ordered groups.  The factor \(K_H/K_H^-\) is
nontrivial exactly when
\[
 H=S_G(\e,k)
\]
for some \(k\in K\setminus\{\e\}\), and the order induced on \(K\) is
recovered from these nontrivial factors by the same leading-layer rule
as the order on \(G\).  In particular, if \(G\) is \(N\)-free, every
nontrivial inherited factor has the same \(\Ctype/\Atype\)-type as the
corresponding ambient factor \(H/H^-\).
\end{corollary}

\begin{proof}
Since \(\cV(G)\) is a chain, the family \(\{K\cap H:H\in\cV(G)\}\)
is a chain of subgroups of \(K\).  Fix \(H\in\cV(G)\).  By
\cref{thm:intro-gallai}, \(H\) is a subgroup-module of \(G\), so
\cref{thm:modular-subgroup-overlap}\textup{(a)} gives that \(K_H=K\cap H\)
is a convex subgroup and a module of the induced order on \(K\).

The subgroup \(H^-\) is a subgroup-module of the ordered group \(H\),
and \(H^-\triangleleft H\)
by \cref{thm:intro-gallai}\textup{(b)}.  Applying
\cref{thm:modular-subgroup-overlap}\textup{(d)} inside \(H\), with
subgroup \(K_H\), gives \(K_H\cap H^- =K\cap H^- =K_H^-,\)
and hence \(K_H^-\triangleleft K_H.\)
It also gives an embedding of ordered groups
\[
 K_H/K_H^-\longrightarrow H/H^-,
 \qquad
 kK_H^-\longmapsto kH^-.
\]

Here \(K_H^-\) denotes the intersection \(K\cap H^-\), relative to the
ambient canonical level \(H\); no assertion is being made at this stage
that it is the intrinsic lower canonical subgroup of \(K_H\).

The factor \(K_H/K_H^-\) is nontrivial exactly when \(K\cap(H\setminus H^-)\ne\varnothing.\)
If \(k\in K\cap(H\setminus H^-)\), then \(S_G(\e,k)\leq H\)
because \(H\) is strong and contains \(\e\) and \(k\).  If the
inclusion were strict, then \(S_G(\e,k)<H,\)
so \(k\in H^-\), a contradiction.  Thus \(S_G(\e,k)=H.\)
Conversely, if \(k\in K\setminus\{\e\}\) and \(S_G(\e,k)=H,\)
then \(k\in H\setminus H^-\), and therefore
\(K_H/K_H^-\) is nontrivial.  Hence the active ambient levels are
precisely \(\{S_G(\e,k):k\in K\setminus\{\e\}\}.\)

Finally, let \(x,y\in K\) be distinct and put \(H:=S_G(\e,x^{-1}y).\)
Since \(K\) is a subgroup, \(x^{-1}y\in K.\)
By the definition of \(H\), \(x^{-1}y\in H\setminus H^-,\)
and hence \(x^{-1}y\in K_H\setminus K_H^-.\)
By \cref{thm:intro-gallai}\textup{(c)}, the comparison of \(x\) and
\(y\) is determined by the image of \(x^{-1}y\) in \(H/H^-\).
Since \(K_H/K_H^-\longrightarrow H/H^-\)
is an order embedding, the same comparison is determined by the image
of \(x^{-1}y\) in \(K_H/K_H^-\).  This is precisely the inherited
leading-layer rule.

If \(G\) is \(N\)-free, every ambient factor \(H/H^-\) is either
totally ordered or equality-ordered.  Every nontrivial ordered subgroup
of a totally ordered factor is again totally ordered, while every
subgroup of an equality-ordered factor is equality-ordered.  Thus every
nontrivial inherited factor has the same \(\Ctype/\Atype\)-type as its
ambient factor.
\end{proof}

\section{Further consequences of the canonical decomposition}
\label{sec:canonical-consequences}

The canonical chain has an immediate finite-generation consequence.
If \(G\) is a nontrivial finitely generated ordered group, choose a finite
generating set \(X\subseteq G\setminus\{\e\}.\)
Since \(\cV(G)\) is a chain, the finitely many values \(S_G(\e,x), \qquad x\in X,\)
have a largest member.  This member contains \(X\), hence equals \(G\).
Thus \(G\in\cV(G).\)
Consequently, if \(G\) is moreover \(N\)-free and has no nontrivial proper
convex normal subgroup, its order is either total or equality.  This applies
in particular to nontrivial finitely generated algebraically simple groups.

\begin{theorem}[Finiteness of the canonical chain]
\label{thm:canonical-chain-finite}
Let \(G\) be an ordered group whose underlying order has no infinite
antichain.  Then \(\cV(G)\) is finite.  In particular, every ordered group
of finite width has a finite canonical chain.
\end{theorem}

\begin{proof}
Let
\[
 \mathcal N
 :=
 \{H\in\cV(G):Q_H=H/H^-\text{ is not totally ordered}\}.
\]
We first show that \(\mathcal N\) is finite.  Fix \(H\in\mathcal N\), and
consider the point \(H^-\) of the Gallai quotient \(Q_H\).  We claim that
it has an incomparable point in \(Q_H\).

If \(Q_H\) is equality-ordered, this is immediate since \(Q_H\) has at
least two points.  Otherwise \(Q_H\) is prime.  Suppose, for a contradiction,
that \(H^-\) is comparable with every point of \(Q_H\).  Its strict lower
set and strict upper set are then modules of \(Q_H\): every point outside
the strict lower set is above all of it, while every point outside the
strict upper set is below all of it.  Both are proper modules, since neither
contains \(H^-\).  Since \(Q_H\) is prime, each is therefore empty or a
singleton.  Consequently \(Q_H\) is a chain with at most three elements,
contradicting the fact that \(H\in\mathcal N\).  Thus there is a Gallai
block of \(H\) incomparable with \(H^-\).

Choose a representative \(x_H\) in that block.  Since relations between
Gallai blocks are uniform,
\[
 x_H\parallel y
 \qquad
 \text{for every }y\in H^-.
\tag{*}
\]

If \(\mathcal N\) were infinite, choose such an \(x_H\) for every
\(H\in\mathcal N\).  For distinct \(H,K\in\mathcal N\), say \(H<K\), one
has \(H\subseteq K^-,\)
and hence \(x_H\in K^-\).  By \((*)\), \(x_K\parallel x_H.\)
Therefore \(\{x_H:H\in\mathcal N\}\)
would be an infinite antichain, a contradiction.  Hence
\(\mathcal N\) is finite.

It remains to control the levels with totally ordered quotients.  If
\(H<K\) are two such levels, then they are nested nonsingleton robust
modules whose Gallai quotients are both linear orders.  By
\cref{prop:HPW-alternation}, there is a robust module \(L\) with \(H<L<K\)
whose Gallai quotient is not a linear order.  Since \(\e\in H\subset L,\)
\cref{lem:robust-identity-branch} gives \(L\in\cV(G)\).  Hence
\(L\in\mathcal N\).

Thus between any two canonical levels with totally ordered quotients lies
a member of the finite set \(\mathcal N\).  Removing the
\(|\mathcal N|\) members of \(\mathcal N\) from the chain \(\cV(G)\)
leaves at most \(|\mathcal N|+1\) convex intervals, and each such interval
contains at most one level with totally ordered quotient.  Hence there are
at most \(|\mathcal N|+1\) such levels.  Therefore \(\cV(G)\) is finite.
\end{proof}

\begin{corollary}[Finite-width factorization]
\label{cor:general-width-factorization}
An ordered group \(G\) has finite width if and only if \(\cV(G)\) is finite
and every canonical factor \(Q_H=H/H^-\), \(H\in\cV(G)\), has finite width.
In that case
\[
 \width(G)
 =
 \prod_{H\in\cV(G)}\width(Q_H),
\]
with the empty product interpreted as \(1\).
\end{corollary}

\begin{proof}
If \(G\) has finite width, \cref{thm:canonical-chain-finite} gives that
\(\cV(G)\) is finite.  Every canonical factor has finite width, since an
antichain of Gallai blocks yields, by choosing one representative from
each block, an antichain in \(G\).

Conversely, the trivial group is immediate, so assume \(G\ne\{\e\}.\)
Suppose \(\cV(G)=\{H_1<\cdots<H_r\}.\)
Since every \(g\ne\e\) belongs to its value \(v(g)=S(\e,g)\in\cV(G),\)
the largest canonical level contains every element of \(G\).  Hence \(H_r=G.\)
Moreover, \(H_1^-=\{\e\},\)
and \(H_i^-=H_{i-1} \qquad (2\leq i\leq r).\)

Since the Gallai blocks of \(H_1\) are singletons, \(\width(H_1)=\width(Q_{H_1}).\)
For \(2\leq i\leq r\), the ordered set \(H_i\) is the
lexicographical sum of copies of \(H_{i-1}\), indexed by \(Q_{H_i}\).

If a poset is the lexicographical sum of copies of a finite-width poset
\(P\), indexed by a finite-width poset \(Q\), then its width is \(\width(P)\width(Q).\)
Indeed, an antichain meets only blocks indexed by an antichain of \(Q\),
and contains at most \(\width(P)\) elements in each such block.
Conversely, choose a maximum antichain \(A\) of \(Q\), and in each block
indexed by \(q\in A\) choose a maximum antichain of that copy of \(P\).
Their union is an antichain of size \(\width(P)\width(Q).\)

Hence, for \(2\leq i\leq r\),
\[
 \width(H_i)
 =
 \width(H_{i-1})\,\width(Q_{H_i}).
\]
Induction gives
\[
 \width(G)
 =
 \width(H_r)
 =
 \prod_{i=1}^r\width(Q_{H_i}),
\]
which is the required product formula.  In particular, \(G\) has finite
width.
\end{proof}

\begin{lemma}[Leading-value law]
\label{lem:leading-value}
Let \(G\) be an ordered group and let \(x,y\ne\e\).  If
\[
 v(x)<v(y)=H,
\]
then
\[
 v(xy)=v(yx)=H
\]
and
\[
 xyH^-=yH^-,
 \qquad
 yxH^-=yH^-.
\]
Consequently, if \(v(x)\ne v(y)\), then
\[
 v(xy)=\max\{v(x),v(y)\}.
\]
For arbitrary \(x,y\ne\e\) with \(xy\ne\e\),
\[
 v(xy)\leq\max\{v(x),v(y)\}.
\]
Moreover,
\[
 v(x^{-1})=v(x).
\]
\end{lemma}

\begin{proof}
Suppose \(v(x)<v(y)=H.\)
Then \(x\in H^-.\)
Moreover, \(y\in H\setminus H^-.\)
Indeed, if \(y\in H^-\), then \(y\) belongs to some canonical level
\(K<H\), and hence \(v(y)=S(\e,y)\subseteq K<H,\)
contrary to \(v(y)=H\).

Since \(x,y\in H\), both \(xy\) and \(yx\) belong to \(H\).  If
\(xy\in H^-\), then \(y=x^{-1}(xy)\in H^-,\)
a contradiction.  Hence \(xy\in H\setminus H^-.\)
In particular \(xy\ne\e\).  Since \(H\) is strong and contains
\(\e\) and \(xy\), \(v(xy)=S(\e,xy)\subseteq H.\)
If \(v(xy)<H\), then \(xy\in v(xy)\subseteq H^-,\)
again a contradiction.  Thus \(v(xy)=H.\)

Since \(H^-\triangleleft H\) and \(x\in H^-\), \(xyH^- = y(y^{-1}xy)H^- = yH^-.\)

Similarly, if \(yx\in H^-\), then \(y=(yx)x^{-1}\in H^-,\)
a contradiction.  Hence \(yx\in H\setminus H^-,\)
and therefore \(v(yx)=H.\)
Since \(x\in H^-\), \(yxH^-=yH^-.\)

Interchanging \(x\) and \(y\) gives the corresponding conclusion when \(v(y)<v(x).\)
Therefore, whenever \(v(x)\ne v(y)\), \(v(xy)=\max\{v(x),v(y)\}.\)

Now suppose \(v(x)=v(y)=H\)
and \(xy\ne\e.\)
Since \(x,y\in H\), we have \(xy\in H\).  As \(H\) is strong and
contains \(\e\) and \(xy\), \(v(xy)=S(\e,xy)\subseteq H.\)
Thus, for arbitrary \(x,y\ne\e\) with \(xy\ne\e\), \(v(xy)\leq\max\{v(x),v(y)\}.\)

Finally, by \eqref{eq:inverse-S}, \(S(\e,x^{-1}) = S(\e,x)^{-1} = S(\e,x),\)
the last equality holding because \(S(\e,x)\) is a subgroup.  Therefore \(v(x^{-1})=v(x).\)
\end{proof}

\subsection{The two graph kernels}
\label{sec:graph-kernels}

The total and equality quotient types admit intrinsic group-theoretic
subgroups detected by the incomparability and comparability graphs.
For an ordered group \(G\), let \(I(G)(\e)\) be the connected component
of \(\e\) in the incomparability graph \(\Inc(G)\).  We use the following
result of Pouzet and Zaguia, assembled from
\cite[Theorem~15(b) and the accompanying discussion]{PouzetZaguia2019}.

\begin{theorem}[Pouzet--Zaguia]\label{thm:I-normal}
For every ordered group \(G\), \(I(G)(\e)\) is the least subgroup of
\(G\) containing \(\inc(\e)\).  It is a convex normal subgroup, the
quotient \(G/I(G)(\e)\) is totally ordered, and the order on \(G\) is
the lexicographical sum of copies of \(I(G)(\e)\) indexed by the chain
\(G/I(G)(\e)\).  In particular,
\[
 I(G)(\e)=\langle\inc(\e)\rangle.
\]
\end{theorem}

There is an elementary counterpart for the comparability graph.  Let
\(C(G)(\e)\) denote the connected component of \(\e\) in the
comparability graph \(\Comp(G)\).

\begin{proposition}[Comparability kernel]\label{prop:C-normal}
For every ordered group \(G\),
\[
 C(G)(\e)
 =
 \left\langle
 G\setminus\bigl(\inc(\e)\cup\{\e\}\bigr)
 \right\rangle.
\]
Moreover, \(C(G)(\e)\) is a convex normal subgroup of \(G\), the
quotient \(G/C(G)(\e)\) is equality-ordered, and the order on \(G\) is
the lexicographical sum of copies of \(C(G)(\e)\) indexed by that
antichain.
\end{proposition}

\begin{proof}
For \(x,y\in G\), the vertices \(x\) and \(y\) are adjacent in
\(\Comp(G)\) exactly when
\[
 x^{-1}y
 \in
 G\setminus\bigl(\inc(\e)\cup\{\e\}\bigr).
\]
Hence the connected component of \(\e\) is precisely the subgroup
generated by this set.

The set \(G\setminus\bigl(\inc(\e)\cup\{\e\}\bigr)\)
is closed under inversion and invariant under conjugation: inversion is
an order anti-automorphism, while conjugation is an order automorphism
fixing \(\e\).  Therefore the subgroup that it generates is normal in
\(G\).  Thus \(C(G)(\e)\triangleleft G.\)

If \(a,b\in C(G)(\e)\) and \(a\leq z\leq b,\)
then either \(z=a\), or \(z\) is adjacent to \(a\) in \(\Comp(G)\).
In either case \(z\in C(G)(\e)\).  Thus \(C(G)(\e)\) is convex.

If \(x\notin C(G)(\e)\) and \(c\in C(G)(\e)\), then \(x\) and \(c\)
cannot be comparable, since otherwise they would be adjacent and hence
belong to the same connected component.  Thus distinct comparability
components are pairwise incomparable.

Since left translations are automorphisms of \(\Comp(G)\), its connected
components are precisely the left cosets of \(C(G)(\e)\); by normality
these are also its right cosets.  Since \(C(G)(\e)\) is convex and
normal, \cref{convex} equips \(G/C(G)(\e)\) with its quotient order.
Distinct cosets are pairwise incomparable, so this quotient order is the
equality order.  Each coset is order-isomorphic to \(C(G)(\e)\), and
therefore the order on \(G\) is the lexicographical sum of copies of
\(C(G)(\e)\) indexed by that antichain.
\end{proof}

Applied to a canonical level \(H\in\cV(G)\), these two kernels refine the
total and equality alternatives in \cref{thm:intro-gallai}.  If \(H/H^-\)
is totally ordered, then \(\inc_H(\e)\subseteq H^-,\)
and hence, by \cref{thm:I-normal} applied to the ordered group induced on
\(H\), \(I(H)(\e)\subseteq H^-.\)
If \(H/H^-\) is equality-ordered, then every nonidentity element of \(H\)
comparable with \(\e\) lies in \(H^-\).  Therefore
\cref{prop:C-normal}, again applied inside \(H\), gives \(C(H)(\e)\subseteq H^-.\)

In the \(N\)-free case, the density of the \(0/1\)-tree together with the
leading-value law upgrades both inclusions to equalities; this is
\cref{prop:canonical-spine} below.

\section{The \texorpdfstring{\(N\)}{N}-free case}
\label{sec:nfree}

For \(N\)-free ordered groups, \cref{thm:HPW-tree} gives the robust
identity branch its dense \(0/1\)-valuation.  The following proposition
translates this structure into ordered-group language and identifies the
canonical subgroups \(H^-\) intrinsically.

\begin{proposition}[The canonical \(0/1\)-spine and the subgroups \(H^-\)]
\label{prop:canonical-spine}
Let \(G\) be an \(N\)-free ordered group.  Then every canonical quotient
\(H/H^-\) is totally ordered or equality-ordered.  Define
\[
 \lambda(H)=
 \begin{cases}
  \Ctype,&H/H^-\text{ is totally ordered},\\
  \Atype,&H/H^-\text{ is equality-ordered}.
 \end{cases}
\]
Then the colouring
\[
 \lambda:\cV(G)\longrightarrow\{\Ctype,\Atype\}
\]
is reduced, and
\begin{equation}
 H^-=
 \begin{cases}
  I(H)(\e),&\lambda(H)=\Ctype,\\
  C(H)(\e),&\lambda(H)=\Atype,
 \end{cases}
 \label{eq:intrinsic-lower-subgroup}
\end{equation}
where the two components are computed in the ordered group induced on
\(H\).  Equivalently,
\[
 H^-=\langle\inc_H(\e)\rangle
 \quad\text{at a \(\Ctype\)-level},
\]
and
\[
 H^-=
 \left\langle
 \{h\in H\setminus\{\e\}:h\text{ is comparable with }\e\}
 \right\rangle
 \quad\text{at an \(\Atype\)-level}.
\]
Moreover, for \(g\ne\e\), with \(H=S(\e,g)\),
\begin{equation}
 g>\e
 \quad\Longleftrightarrow\quad
 \lambda(H)=\Ctype
 \text{ and }gH^->H^-
 \text{ in }H/H^-.
 \label{eq:canonical-leading}
\end{equation}
If \(\lambda(H)=\Atype\), then \(g\parallel\e\).
\end{proposition}

\begin{proof}
By \cref{thm:HPW-tree}, the Gallai quotient at a nonsingleton robust
node is a chain at a \(1\)-node and an antichain at a \(0\)-node.
By \cref{thm:intro-gallai}\textup{(b)}, these quotients carry their
natural ordered-group structures.  Thus every canonical quotient is
totally ordered or equality-ordered.  By
\cref{lem:robust-identity-branch}, the nonsingleton robust nodes on the
branch through \(\e\) are exactly the members of \(\cV(G)\).

We first prove that the colouring is reduced.  Suppose that
\(H<K\) are canonical levels of the same type.  Their corresponding
robust nodes have the same \(0/1\)-label.  Density in
\cref{thm:HPW-tree} therefore gives a robust module \(L\) with \(H<L<K\)
of the opposite label.  Since \(\e\in H\subset L\),
\cref{lem:robust-identity-branch} gives \(L\in\cV(G)\).  Hence
\(\lambda\) is reduced.

Fix \(H\in\cV(G)\), and put
\[
 \Gamma_H=
 \begin{cases}
  \Inc(H),&\lambda(H)=\Ctype,\\
  \Comp(H),&\lambda(H)=\Atype.
 \end{cases}
\]
The preceding graph-kernel discussion already established that the
component of \(\e\) in \(\Gamma_H\) is contained in \(H^-\).  We prove
the opposite inclusion.

Let \(x\in H^-\setminus\{\e\}, \qquad L:=v(x)<H.\)
If \(\lambda(L)\ne\lambda(H)\), then the general leading-layer rule
shows that \(x\) is adjacent to \(\e\) in \(\Gamma_H\), and hence lies
in its component.

Suppose instead that \(\lambda(L)=\lambda(H).\)
Since the colouring is reduced, there is \(M\in\cV(G)\) such that \(L<M<H, \qquad \lambda(M)\ne\lambda(H).\)
Choose \(y\in M\setminus M^-\).  Then \(v(y)=M\), so the general
leading-layer rule shows that \(y\) is adjacent to \(\e\) in
\(\Gamma_H\).  Moreover, \(v(x^{-1})=v(x)=L<M=v(y),\)
and therefore \cref{lem:leading-value} gives \(v(x^{-1}y)=M.\)
Again by the general leading-layer rule, \(x^{-1}y\) has the relation
to \(\e\) corresponding to \(\Gamma_H\); equivalently, \(x\) is
adjacent to \(y\) in \(\Gamma_H\).  Hence \(\e-y-x\)
is a path in \(\Gamma_H\), and \(x\) lies in the component of \(\e\).
Consequently
\[
 H^-=
 \begin{cases}
  I(H)(\e),&\lambda(H)=\Ctype,\\
  C(H)(\e),&\lambda(H)=\Atype.
 \end{cases}
\]

The equivalent generated-subgroup descriptions now follow from
\cref{thm:I-normal,prop:C-normal}.

Finally, let \(g\ne\e\) and put \(H=S(\e,g)\).  The equivalence
\eqref{eq:canonical-leading} is precisely the general leading-layer
rule \eqref{eq:general-leading}.  If \(\lambda(H)=\Atype\), then
\(H/H^-\) is equality-ordered and \(gH^-\ne H^-,\)
so \(g\parallel\e\).
\end{proof}

For the converse, we axiomatize the subgroup-chain structure from which the
order is reconstructed.

\begin{definition}[Admissible two-coloured subgroup chain]
\label{def:admissible-chain}
An admissible two-coloured subgroup chain on a nontrivial group \(G\)
consists of a chain \(\cH\) of subgroups together with a map
\[
 \lambda:\cH\longrightarrow\{\Ctype,\Atype\}
\]
satisfying:
\begin{enumerate}[label=\textup{(D\arabic*)}]
\item for every \(g\ne\e\), the set
\[
 \{H\in\cH:g\in H\}
\]
has a least member, denoted \(v(g)\);

\item for every \(H\in\cH\),
\[
 H^-:=\{\e\}\cup\bigcup_{K<H}K
\]
is a proper normal subgroup of \(H\);

\item if \(\lambda(H)=\Ctype\), then \(H/H^-\) is equipped with a
compatible total group order, while if \(\lambda(H)=\Atype\), then
\(H/H^-\) has the equality order;

\item conjugation by every element of \(G\) permutes \(\cH\), preserves
inclusion and the colouring \(\lambda\), and induces an order
isomorphism between the corresponding type \(\Ctype\) quotients.
\end{enumerate}
The chain is \emph{reduced} if, whenever \(K<H\) have the same type,
there is \(L\) with
\[
 K<L<H
\]
of the opposite type.
\end{definition}

For \(g\ne\e\), with \(H=v(g)\), declare
\begin{equation}
 g>\e
 \quad\Longleftrightarrow\quad
 \lambda(H)=\Ctype
 \text{ and }gH^->H^-
 \text{ in }H/H^-.
 \label{eq:constructed-positive}
\end{equation}
If \(\lambda(H)=\Atype\), declare \(g\parallel\e\).  For arbitrary
\(x,y\in G\), set
\begin{equation}
 x<y
 \quad\Longleftrightarrow\quad
 x^{-1}y>\e.
 \label{eq:constructed-order}
\end{equation}

The link with the cograph tree is the following elementary coset
construction.

\begin{lemma}[The coset meet-tree]\label{lem:coset-tree}
Let \(\cH\) be a reduced admissible two-coloured subgroup chain on
\(G\).  Let
\[
 \mathcal T(\cH)
 :=
 \{\{g\}:g\in G\}
 \cup
 \{gH:g\in G,\ H\in\cH\},
\]
ordered by reverse inclusion, and label a nonsingleton node \(gH\) by
\(1\) when \(\lambda(H)=\Ctype\) and by \(0\) when
\(\lambda(H)=\Atype\).  Then \(\mathcal T(\cH)\) is a densely
\(0/1\)-valued ramified meet-tree.  For distinct \(x,y\in G\), if
\[
 H=v(x^{-1}y),
\]
then
\begin{equation}
 \{x\}\wedge\{y\}=xH.
 \label{eq:coset-meet}
\end{equation}
\end{lemma}

\begin{proof}
If two cosets of members of \(\cH\) meet, then, because their
subgroups are comparable, the cosets are nested.  Hence
\(\mathcal T(\cH)\) is a forest in the reverse-inclusion order.

We verify that every two nodes have a meet.  The case of two nested
nodes is immediate.  Suppose that \(xH\) and \(yK\) are disjoint; the
same argument covers a singleton by regarding its level formally as
\(\{\e\}\).  Put \(L:=v(x^{-1}y).\)
We must have \(H\leq L\).  Indeed, if \(L<H\), then
\(x^{-1}y\in H\), and hence \(y\in xH\), a contradiction.  Likewise
\(K\leq L\): if \(L<K\), then \(y^{-1}x=(x^{-1}y)^{-1}\in K,\)
so \(x\in yK\), again a contradiction.

Thus \(xL\) contains both \(xH\) and \(yK\).  If a coset \(zM\)
contains both, then \(x,y\in zM\), so \(x^{-1}y\in M\)
and hence \(L\leq M\).  Since \(x\in zM\), we have \(zM=xM,\)
and therefore \(xL\subseteq zM.\)
Thus \(xL\) is the smallest coset, under ordinary inclusion, containing
the two nodes, and hence their meet in the reverse-inclusion order.  In
particular, for leaves \(x\ne y\) this gives
\eqref{eq:coset-meet}.

Every nonsingleton node \(gH\) is itself the meet of two leaves.
Indeed, by (D2) choose \(t\in H\setminus H^-.\)
Then (D1) gives \(v(t)=H\), and hence \(gH=\{g\}\wedge\{gt\}.\)
Thus the meet-tree is ramified.

The label of a coset is well defined.  If two nonempty left cosets are
equal, translating their common set by the inverse of any one of its
points recovers the underlying subgroup, so the two levels are equal.

Finally, let two nested nonsingleton nodes have the same label.  Their
underlying levels have the same type.  Reducedness supplies an
intermediate level of the opposite type, and the unique coset of that
level containing the smaller node is an intermediate node with the
opposite label.  Hence the valuation is dense.
\end{proof}

\begin{theorem}[Leading-layer representation]
\label{thm:leading-representation}
Let \(\cH\) be a reduced admissible two-coloured subgroup chain on a
nontrivial group \(G\).  Then
\eqref{eq:constructed-positive}--\eqref{eq:constructed-order} define an
\(N\)-free partial order on \(G\) invariant under multiplication on
both sides.  Its comparability graph has valued robust-module tree
\(\mathcal T(\cH)\).  In particular, every \(H\in\cH\) is a convex
strong module,
\begin{equation}
 S(\e,g)=v(g)
 \qquad(g\ne\e),
 \label{eq:canonical-value}
\end{equation}
and the quotient order on \(H/H^-\) is the prescribed total or equality
order.

Conversely, the canonical chain of every nontrivial \(N\)-free ordered
group is a reduced admissible two-coloured subgroup chain, and the
leading-layer construction recovers the original order.
\end{theorem}

\begin{proof}
Let \(P\) be the set of elements declared strictly positive.  We first
record the leading-value calculation.  Suppose \(v(x)<v(y)=H.\)
Then \(x\in H^-, \qquad y\in H\setminus H^-.\)
Since \(H^-\) is a subgroup, neither \(xy\) nor \(yx\) lies in \(H^-\);
otherwise \(y\in H^-\).  Thus \(xy,yx\in H\setminus H^-,\)
and therefore \(v(xy)=v(yx)=H.\)
Since \(H^-\triangleleft H\), \(xyH^-=yH^-, \qquad yxH^-=yH^-.\)
Consequently, whenever \(v(x)\ne v(y)\),
\begin{equation}
 v(xy)=\max\{v(x),v(y)\},
 \label{eq:constructed-leading-value}
\end{equation}
and the leading coset of the product is the leading coset of the factor
with larger value.

Also,
\begin{equation}
 v(x^{-1})=v(x),
 \label{eq:constructed-inverse-value}
\end{equation}
because \(x\) and \(x^{-1}\) belong to exactly the same members of the
subgroup chain \(\cH\).

We verify the positive-cone conditions.  Let \(x,y\in P\).  If
\(v(x)\ne v(y)\), then
\eqref{eq:constructed-leading-value} shows that \(xy\) has the leading
coset, and hence the sign, of the factor with larger value.

Suppose \(v(x)=v(y)=H.\)
Then \(H\) has type \(\Ctype\), and the two nonidentity cosets \(xH^-, \qquad yH^-\)
are positive in the totally ordered group \(H/H^-\).  Their product \(xyH^-\)
is therefore positive, and in particular is not the identity coset.
Hence \(xy\in H\setminus H^-, \qquad v(xy)=H,\)
and \(xy\in P\).  Thus \(PP\subseteq P.\)

By \eqref{eq:constructed-inverse-value}, the leading coset of \(g^{-1}\)
is the inverse of the leading coset of \(g\).  Hence \(P\cap P^{-1}=\varnothing.\)
Condition (D4) gives \(aPa^{-1}=P \qquad(a\in G).\)
Therefore \eqref{eq:constructed-order} defines a strict partial order.
Left invariance is immediate, and conjugation invariance of \(P\) gives
right invariance.

Consider now the coset tree \(\mathcal T(\cH)\) of
\cref{lem:coset-tree}.  For distinct \(x,y\in G\), let \(H=v(x^{-1}y).\)
At a type \(\Ctype\) level, the nonidentity coset \(x^{-1}yH^-\)
is either positive or negative in the prescribed total order, whereas
at a type \(\Atype\) level it is incomparable with the identity coset.
Hence
\[
 x\text{ and }y\text{ are comparable}
 \quad\Longleftrightarrow\quad
 \lambda(H)=\Ctype
 \quad\Longleftrightarrow\quad
 \tau(\{x\}\wedge\{y\})=1.
\]
Thus the comparability graph of the constructed order is exactly the
graph associated with the densely valued tree \(\mathcal T(\cH)\).
By \cref{thm:HPW-tree}, it is a cograph and its valued robust-module
tree is naturally isomorphic to \(\mathcal T(\cH)\).  Hence the order
is \(N\)-free.

Under the natural identification of the leaves with \(G\), the nodes
\(gH\) correspond to the robust modules of the constructed
comparability graph.  Since a poset and its comparability graph have the
same strong modules, they are also the corresponding robust modules of
the constructed order.  In particular every \(H\in\cH\) is strong,
hence convex by \cref{lem:1}.  Moreover, for \(g\ne\e\),
\eqref{eq:coset-meet} with \(x=\e\) shows that the least robust module
containing \(\e\) and \(g\) is \(v(g)\).  Thus \(S(\e,g)=v(g),\)
proving \eqref{eq:canonical-value}.

It remains only to identify the quotient order.  If \(x,y\in H\) lie
in distinct \(H^-\)-cosets, then \(x^{-1}y\in H\setminus H^-,\)
so \(v(x^{-1}y)=H.\)
Therefore \eqref{eq:constructed-positive} gives the prescribed total
order on \(H/H^-\) at a \(\Ctype\)-level, while distinct cosets are
incomparable at an \(\Atype\)-level.

Conversely, let \(G\) be a nontrivial \(N\)-free ordered group.  By
\cref{thm:intro-gallai}, its canonical chain satisfies (D1), (D2), and
the conjugation requirements in (D4).  By
\cref{prop:canonical-spine}, its quotients have the prescribed two
types and its colouring is reduced.  Thus it is a reduced admissible
two-coloured subgroup chain, and \eqref{eq:canonical-leading} recovers
the original order.
\end{proof}

\begin{proof}[Proof of \cref{thm:intro-characterisation}]
For an \(N\)-free ordered group,
\cref{prop:canonical-spine,thm:intro-gallai} show that the canonical
chain is a reduced admissible two-coloured subgroup chain, and the
leading-layer rule recovers the given order.  Conversely,
\cref{thm:leading-representation} shows that every reduced admissible
two-coloured subgroup chain yields an \(N\)-free ordered group.
Its identity \eqref{eq:canonical-value} identifies the prescribed
levels with the canonical values \(S(\e,g)\), showing that the
prescribed chain is exactly the canonical chain.
\end{proof}

\begin{proof}[Proof of \cref{thm:intro-subgroup-restriction}]
The assertion is immediate for \(K=\{\e\}\), so assume that \(K\ne\{\e\}.\)
Write
\[
 v(k):=S_G(\e,k)
 \qquad(k\in K\setminus\{\e\}),
\]
and let \(\mathcal A_K\), \(\mathcal B_K\), and \(M_B\) be as in the
statement.  We order \(\mathcal B_K\) by the order inherited from the
convex blocks of the chain \(\mathcal A_K\).

We first show that every \(M_B\) is a subgroup of \(K\).  Take
\(x,y\in M_B\).  If \(x=\e\), then \(x^{-1}y=y\in M_B.\)
If \(y=\e\), then \(x^{-1}y=x^{-1},\)
and \cref{lem:leading-value} gives \(v(x^{-1})=v(x),\)
so \(x^{-1}\in M_B\).  If \(x=y\), there is nothing to prove.

Otherwise \(x,y\ne\e\), and both lie in the larger, under inclusion, of
their ambient value subgroups.  By \cref{lem:leading-value}, \(v(x^{-1}y) \leq \max\{v(x),v(y)\}.\)
Since \(x^{-1}y\in K\setminus\{\e\}\), its value is active, and its
active block is at most \(B\).  Hence \(x^{-1}y\in M_B.\)
Thus \(M_B\leq K.\)

For \(B\in\mathcal B_K\), put
\begin{equation}
 M_B^-:=
 \{\e\}\cup\bigcup_{C<B}M_C
 =
 \{\e\}\cup
 \{k\in K\setminus\{\e\}:[v(k)]<B\}.
 \label{eq:restricted-lower-subgroup}
\end{equation}
This is a subgroup, being the union of a chain of subgroups, and it is
proper in \(M_B\) because the block \(B\) is active.

We claim that \(M_B^-\) is a module of the induced order on \(M_B\).
Let \(z\in M_B\setminus M_B^-, \qquad h\in M_B^-.\)
If \(h=\e\), there is nothing to prove.  Otherwise \([v(h)]<B=[v(z)],\)
and therefore \(v(h)<v(z).\)
By \cref{lem:leading-value}, \(h^{-1}z\,v(z)^- = z\,v(z)^-.\)
Thus \(h^{-1}z\) and \(z\) have the same leading coset.  If the block
\(B\) has type \(\Ctype\), they therefore have the same sign relative
to \(\e\); if \(B\) has type \(\Atype\), both are incomparable with
\(\e\).  Equivalently,
\[
 h<z\quad\Longleftrightarrow\quad\e<z,
 \qquad
 z<h\quad\Longleftrightarrow\quad z<\e.
\]
Hence every point of \(M_B\setminus M_B^-\) sees all of \(M_B^-\)
uniformly in both directions.  Thus \(M_B^-\) is a module of \(M_B\),
and \cref{lem:1} gives that it is convex.

We prove normality next.  Conjugation by any \(a\in K\) acts on the
ambient canonical chain by a type-preserving order automorphism, and
\[
 v(aka^{-1})
 =
 a\,v(k)\,a^{-1}
 \qquad(k\ne\e).
\]
Hence it preserves \(\mathcal A_K\) and permutes its maximal convex
same-type blocks.

Fix \(a\in M_B\).  If \(a=\e\), there is nothing to prove.  If \([v(a)]<B,\)
then \(v(a)<H\) for every \(H\in B\), so \(a\in H\) and therefore \(aHa^{-1}=H \qquad(H\in B).\)
If instead \([v(a)]=B,\)
then \(a\in v(a)\), whence \(av(a)a^{-1}=v(a).\)
A type-preserving order automorphism of \(\mathcal A_K\) fixing one
member of the maximal convex same-type block \(B\) fixes \(B\)
setwise.  Thus in either case conjugation by \(a\) fixes \(B\) and
preserves the blocks below \(B\).  By
\eqref{eq:restricted-lower-subgroup}, \(aM_B^-a^{-1}=M_B^-.\)
Hence \(M_B^-\triangleleft M_B.\)

We identify the quotient type.  Let \(xM_B^-\ne yM_B^-\)
and put \(d:=x^{-1}y.\)
Then \(d\in M_B\setminus M_B^-,\)
so \([v(d)]=B.\)
If \(B\) has type \(\Atype\), then \(d\parallel\e\), and hence
\(x\parallel y\).  Thus distinct \(M_B^-\)-cosets are pairwise
incomparable.  If \(B\) has type \(\Ctype\), then \(d\) is comparable
with \(\e\), and hence \(x\) and \(y\) are comparable.  Since
\(M_B^-\) is a convex normal subgroup-module, the quotient order is
well defined; it is equality-ordered in the first case and totally
ordered in the second.  Therefore \(M_B/M_B^-\)
has exactly the common type of the ambient levels in \(B\).

The chain \(\mathcal M_K:=\{M_B:B\in\mathcal B_K\}\)
is reduced.  Indeed, if two distinct blocks of the same type had no
block of the opposite type between them, all intervening blocks would
have that same type, and their union would be a larger convex
constant-type subset of \(\mathcal A_K\), contradicting maximality.

Every \(k\in K\setminus\{\e\}\) belongs to the least member \(M_{[v(k)]}\)
of \(\mathcal M_K\), so (D1) holds.  The preceding arguments give
(D2) and (D3).

For (D4), let \(a\in K\).  The type-preserving order automorphism \(H\longmapsto aHa^{-1}\)
of \(\mathcal A_K\) induces an order- and type-preserving permutation
of \(\mathcal B_K\).  Writing \(a\cdot B:=\{aHa^{-1}:H\in B\},\)
we have directly from the definitions
\[
 aM_Ba^{-1}=M_{a\cdot B},
 \qquad
 aM_B^-a^{-1}=M_{a\cdot B}^-.
\]
Since conjugation is an order automorphism of \(K\), it induces the
required order isomorphisms on the type \(\Ctype\) quotients.  Hence
\(\mathcal M_K\) is a reduced admissible two-coloured subgroup chain.

Finally, let \(k\ne\e\) and put \(B:=[v(k)].\)
Then \(k\in M_B\setminus M_B^-,\)
and the quotient order just described records exactly the original
relation of \(k\) to \(\e\).  Therefore the leading-layer order
defined by \(\mathcal M_K\) is precisely the order induced on \(K\).
Applying \cref{thm:leading-representation} to this admissible chain and
using \eqref{eq:canonical-value} gives
\[
 \cV(K)
 =
 \mathcal M_K
 =
 \{M_B:B\in\mathcal B_K\}.
\]

If a block \(B\) has no largest ambient active level, then \(M_B\) is
genuinely the union subgroup in \eqref{eq:intro-restricted-level}; this is
the limit case absent from the finite pruning argument.
\end{proof}

\begin{corollary}[Normal subgroup-module quotients]
\label{cor:nfree-modular-quotient}
Let \(G\) be an \(N\)-free ordered group and let
\(K\triangleleft G\) be a proper subgroup-module.  If \(K\) is strong,
its quotient decomposition is given by
\cref{cor:quotient-compatibility}.  Otherwise let \(H\in\cV(G)\) be
the unique level with
\[
 H^-<K<H.
\]
Then \(G/K\) is \(N\)-free and
\[
 \cV(G/K)
 =
 \{H/K\}
 \cup
 \{L/K:L\in\cV(G),\ H<L\}.
\]
The first factor \(H/K\) has the same
\(\Ctype/\Atype\)-type as \(H/H^-\), and every canonical factor
strictly above it is naturally order-isomorphic to the corresponding
factor of \(G\).
\end{corollary}

\begin{proof}
The strong case is \cref{cor:quotient-compatibility}, so suppose that
\(K\) is not strong.  By \cref{thm:intro-subgroup-modules}, there is a
unique \(H\in\cV(G)\) with \(H^-<K<H.\)
Since a subgroup-module is convex and \(K\triangleleft G\), the quotient
\(G/K\) is an ordered group.

Consider the chain
\[
 \mathcal H_K
 :=
 \{H/K\}
 \cup
 \{L/K:L\in\cV(G),\ H<L\}.
\]
Every nonidentity coset \(gK\) belongs to a least member of this chain.
If \(g\in H\setminus K,\)
that member is \(H/K\).  If \(S_G(\e,g)=L>H,\)
it is \(L/K\).  No element of ambient value below \(H\) survives the
quotient, because every such element belongs to \(H^-\subseteq K.\)

With respect to the proposed chain \(\mathcal H_K\), the lower subgroup
of its first level is \((H/K)^-=K/K.\)
Moreover, \(H/K \cong (H/H^-)/(K/H^-).\)
The subgroup \(K/H^-\) is a subgroup-module of \(H/H^-\).  If
\(H/H^-\) is totally ordered, then \(K/H^-\) is convex and the quotient
is totally ordered; if \(H/H^-\) is equality-ordered, then the quotient
is equality-ordered.  Thus \(H/K\) has the same type as \(H/H^-\).

Now let \(L>H\).  Since \(H\subseteq L^-,\)
we have \(K\leq L^-\).  With respect to the proposed chain
\(\mathcal H_K\), the lower subgroup of \(L/K\) is \(L^-/K.\)
Consequently \((L/K)/(L^-/K) \cong L/L^-\)
naturally as ordered groups.

Because \(K\triangleleft G\),
\cref{cor:normal-subgroup-modules} shows that the unique level \(H\)
satisfying \(H^-<K<H\) is normal in \(G\).  Conjugation therefore
descends to \(G/K\), fixes \(H/K\), and permutes the higher retained
levels equivariantly.  Hence \(\mathcal H_K\) satisfies (D4), while the
preceding paragraphs give (D1)--(D3).

The chain is reduced.  Indeed, it is obtained from the ambient reduced
chain by deleting all levels below \(H\) and replacing the first
factor by \(H/K\), which has the same type as \(H/H^-\).  Thus any two
retained levels of the same type still have a retained level of the
opposite type between them.

Finally, the quotient order on \(G/K\) is exactly the leading-layer
order determined by \(\mathcal H_K\).  At the first layer this is the
quotient order on \(H/K\).  Above \(H\), the subgroup \(K\) lies in
every \(L^-\), so the leading coset and its comparison with the identity
are unchanged.

Thus \(\mathcal H_K\) is a reduced admissible two-coloured subgroup
chain whose leading-layer order is the given order on \(G/K\).
Applying \cref{thm:leading-representation} and
\eqref{eq:canonical-value} gives \(\cV(G/K)=\mathcal H_K.\)
In particular \(G/K\) is \(N\)-free, and the stated factor assertions
follow.
\end{proof}

\section{Finite width and intrinsic bottom recognition}
\label{sec:finite-width}

By \cref{thm:canonical-chain-finite,cor:general-width-factorization},
finite width already forces the canonical chain \(\cV(G)\) to be finite
for an arbitrary ordered group.  In the \(N\)-free case the canonical
factors are totally ordered or equality-ordered, so the general
factorization becomes particularly explicit.

\begin{theorem}[Finite-width characterization]
\label{thm:nfree-finite-width}
Let \(G\) be an \(N\)-free ordered group.  The following are equivalent.
\begin{enumerate}[label=\textup{(\roman*)}]
\item \(G\) has no infinite antichain;
\item \(G\) has finite width;
\item the canonical chain \(\cV(G)\) is finite and every type
\(\Atype\) quotient is finite.
\end{enumerate}

Equivalently, there is a finite chain of convex subgroups
\[
 \{\e\}=H_0<H_1<\cdots<H_r=G
\]
such that every \(H_i\) is normal in \(G\), the nontrivial factors
alternate between totally ordered groups and finite equality-ordered
groups, and the order on \(G\) is the leading-factor order.  Moreover,
\[
 \width(G)
 =
 \prod_{H_i/H_{i-1}\text{ of type }\Atype}
 |H_i/H_{i-1}|.
\]
\end{theorem}

\begin{proof}
Assume first that \(G\) has no infinite antichain.  By
\cref{thm:canonical-chain-finite}, the canonical chain is finite.
Every type \(\Atype\) quotient is finite: otherwise, choosing one
representative from each \(H^-\)-coset would give an infinite antichain
in \(G\).  Thus \textup{(i)} implies \textup{(iii)}.

Assume \textup{(iii)}.  A type \(\Ctype\) factor has width \(1\),
while a finite type \(\Atype\) factor has width equal to its
cardinality.  Therefore
\cref{cor:general-width-factorization} gives
\[
 \width(G)
 =
 \prod_{\lambda(H)=\Atype}|Q_H|
 <\infty.
\]
Thus \textup{(iii)} implies \textup{(ii)}.  Finally, finite width
excludes an infinite antichain, so \textup{(ii)} implies
\textup{(i)}.  The same product formula gives the asserted expression
for \(\width(G)\).

For the equivalent finite-chain formulation, assume \textup{(iii)} and
write \(\{\e\}=H_0<H_1<\cdots<H_r=G.\)
Then \(H_i^-=H_{i-1} \qquad(1\leq i\leq r).\)
Conjugation acts by order automorphisms on the finite chain
\(\cV(G)\).  Every order automorphism of a finite chain is the
identity, so
\[
 aH_i a^{-1}=H_i
 \qquad(a\in G,\ 0\leq i\leq r).
\]
Hence every \(H_i\) is normal in \(G\).

By \cref{prop:canonical-spine}, the factors are totally ordered or
equality-ordered.  Under \textup{(iii)}, the equality-ordered factors
are finite.  Since the canonical colouring is reduced and the chain is
finite, the two types alternate.  Finally,
\cref{thm:intro-gallai}\textup{(c)} shows that the given order on
\(G\) is the leading-factor order.

Conversely, suppose that \(\{\e\}=H_0<H_1<\cdots<H_r=G\)
is a finite chain of convex subgroups satisfying the stated
properties.  Put \(H_i^-:=H_{i-1} \qquad(1\leq i\leq r).\)
The top level \(G\) and finiteness of the chain give (D1).
Since every \(H_{i-1}\) is normal in \(G\), it is in particular normal
in \(H_i\), so (D2) holds.  The prescribed factor orders give (D3).
Normality of every level makes conjugation fix the chain, and since
conjugation is an order automorphism of \(G\), it induces an order
automorphism on every type \(\Ctype\) quotient.  Thus (D4) holds.
Alternation gives reducedness.

Hence the chain is a reduced admissible two-coloured subgroup chain.
Since the given order is its leading-factor order,
\cref{thm:leading-representation} shows that this order is \(N\)-free
and that \(\{\e\}=H_0<H_1<\cdots<H_r=G\)
is its canonical chain.
\end{proof}

\begin{corollary}[Width divisibility]
\label{cor:subgroup-width}
Let \(G\) be an \(N\)-free ordered group of finite width, and write its
canonical chain as
\[
 \{\e\}=H_0<H_1<\cdots<H_r=G.
\]
Let \(K\leq G\), and put
\[
 K_i:=K\cap H_i
 \qquad(0\leq i\leq r).
\]
Then
\[
 |\cV(K)|\leq|\cV(G)|,
\]
with the convention \(\cV(\{\e\})=\varnothing\), and
\begin{equation}
 \width(K)
 =
 \prod_{\substack{1\leq i\leq r\\
 H_i/H_{i-1}\text{ of type }\Atype}}
 [K_i:K_{i-1}].
 \label{eq:intro-subgroup-width}
\end{equation}
Consequently,
\[
 \width(K)\mid\width(G).
\]

More generally, if \(L\triangleleft K\) is a subgroup-module of the
induced order on \(K\), then
\[
 \width(K/L)\mid\width(K)\mid\width(G).
\]
\end{corollary}

\begin{proof}
The case \(K=\{\e\}\) is immediate, so suppose \(K\ne\{\e\}.\)
By \cref{thm:intro-subgroup-restriction}, the canonical chain of \(K\)
is obtained from the members of the finite inherited filtration \(K_0\leq K_1\leq\cdots\leq K_r\)
that meet \(K\setminus\{\e\}\), by deleting repetitions and
coalescing consecutive same-type runs.  Hence \(|\cV(K)|\leq|\cV(G)|.\)

For the width formula it is better to retain the unreduced filtration.
Fix \(1\leq i\leq r\).  By subgroup inheritance, the natural map \(K_i/K_{i-1}\longrightarrow H_i/H_{i-1}\)
is an embedding of ordered groups.

If \(H_i/H_{i-1}\) has type \(\Ctype\), then the distinct
\(K_{i-1}\)-cosets in \(K_i\) are linearly ordered.  Therefore \(\width(K_i)=\width(K_{i-1}).\)
If \(H_i/H_{i-1}\) has type \(\Atype\), then those cosets are pairwise
incomparable copies of \(K_{i-1}\), and hence
\[
 \width(K_i)
 =
 [K_i:K_{i-1}]\,\width(K_{i-1}).
\]
Induction gives \eqref{eq:intro-subgroup-width}.

At every type \(\Atype\) step, the ambient factor
\(H_i/H_{i-1}\) is finite.  The group embedding above and Lagrange's
theorem therefore give \([K_i:K_{i-1}] \mid |H_i/H_{i-1}|.\)
Multiplying over all type \(\Atype\) indices and using the width
formula for \(G\) yields \(\width(K)\mid\width(G).\)

Thus the widths of subgroups belong to the divisor lattice of
\(\width(G)\).  In particular, if \(\width(G)=p\) is prime, then
every subgroup is either totally ordered or has width \(p\).

Now let \(L\triangleleft K\) be a subgroup-module of the induced order
on \(K\).  If \(L=K\), then \(K/L\) is trivial, so \(\width(K/L)=1\mid\width(K),\)
and there is nothing to prove.  Hence assume \(L<K.\)

The subgroup \(K\), with its induced order, is again \(N\)-free and
has finite width.  If \(L\) is strong in \(K\),
\cref{cor:quotient-compatibility} deletes an initial part of the
canonical chain of \(K\).  Hence the product of the sizes of the
remaining finite type \(\Atype\) factors divides the full width
product.

Suppose instead that \(L\) is not strong in \(K\).  By
\cref{cor:nfree-modular-quotient}, applied inside \(K\), only the
lowest retained canonical factor changes.  Let \(H\in\cV(K)\) be the
unique level with \(H^-<L<H.\)
If \(H/H^-\) has type \(\Ctype\), both the original and the new
lowest factors contribute width \(1\).  If \(H/H^-\) has type
\(\Atype\), then \(H/H^-\)
is finite and the new lowest factor has cardinality \(|H/L|=[H:L].\)
Since \(H^-\leq L\leq H\), Lagrange's theorem gives \([H:L]\mid[H:H^-].\)
All higher canonical factors are naturally order-isomorphic to their
counterparts in \(K\).  Therefore \(\width(K/L)\mid\width(K).\)
Combining this with the subgroup divisibility already proved gives \(\width(K/L)\mid\width(K)\mid\width(G).\)
\end{proof}

If there are \(a\) type \(\Atype\) factors and \(w:=\width(G)<\infty,\)
then every such factor has at least two elements.  Hence
\[
 2^a
 \leq
 \prod_{\substack{1\leq i\leq r\\
 H_i/H_{i-1}\text{ equality-ordered}}}
 |H_i/H_{i-1}|
 =
 w,
\]
so \(a\leq\lfloor\log_2 w\rfloor.\)
Since the two types alternate, there are at most \(a+1\) type
\(\Ctype\) factors.  Therefore
\begin{equation}
 r
 \leq
 2a+1
 \leq
 2\lfloor\log_2 w\rfloor+1.
 \label{eq:length-bound}
\end{equation}

\subsection{Intrinsic recognition of the least canonical level}

The general compatibility results make the bottom-up recursion largely
formal.  What remains genuinely independent is that, in finite width,
the least canonical level can be recognized directly from the local
comparability and incomparability geometry around the identity.

Recall from \cref{sec:graph-kernels} that \(I(G)(\e)\) denotes the
connected component of \(\e\) in \(\Inc(G)\).  Define
\begin{equation}
 K(G):=
 \{\e\}\cup
 \{x\in\inc(\e):
   \{x\}\text{ is a component of }\Comp(G)[\inc(\e)]\},
 \label{eq:K}
\end{equation}
and
\begin{equation}
 L(G):=
 \{g\in I(G)(\e):
   g\parallel x\text{ for every }x\in\inc(\e)\}.
 \label{eq:L}
\end{equation}

Assume that \(G\) is \(N\)-free of finite width and has at least two
canonical levels \(H_1<\cdots<H_r=G.\)
By \cref{thm:nfree-finite-width}, the levels alternate and every
\(H_i\) is normal in \(G\).

\begin{proposition}[Bottom kernel]
\label{prop:bottom-kernel}
If \(H_1\) has type \(\Atype\), then
\begin{equation}
 K(G)=H_1,
 \qquad
 L(G)=\{\e\}.
 \label{eq:bottom-A}
\end{equation}
If \(H_1\) has type \(\Ctype\), then
\begin{equation}
 K(G)=\{\e\},
 \qquad
 L(G)=H_1.
 \label{eq:bottom-C}
\end{equation}
\end{proposition}

\begin{proof}
Suppose first that \(H_1\) has type \(\Atype\).  Since \(H_1^-= \{\e\},\)
the induced order on \(H_1\) is equality-ordered.  Thus distinct
elements of \(H_1\) are mutually incomparable.

Let
\[
 x\in H_1\setminus\{\e\},
 \qquad
 y\in\inc(\e)\setminus H_1.
\]
The value of \(y\) is a higher type \(\Atype\) level.  Since \(v(x)=H_1<v(y),\)
the leading-value rule shows that the relation between \(x\) and \(y\)
is determined at the type \(\Atype\) level \(v(y)\).  Hence \(x\parallel y.\)
Thus every nonidentity element of \(H_1\) is isolated in
\(\Comp(G)[\inc(\e)]\), and therefore \(H_1\subseteq K(G).\)

Conversely, let \(x\in\inc(\e)\setminus H_1.\)
Its exact value is a higher type \(\Atype\) level.  By alternation,
there is a lower canonical level of type \(\Ctype\).  Choose a
nonidentity element \(c\) of exact value at such a level.  Then \(v(c)<v(x).\)
By the leading-value law, \(v(xc)=v(x),\)
so \(xc\parallel\e.\)
Thus \(xc\in\inc(\e).\)
On the other hand, \(x^{-1}(xc)=c,\)
and \(c\) is comparable with \(\e\).  Hence \(x\) is comparable with
\(xc\).  Since \(c\ne\e\), we have \(xc\ne x\).  Thus \(x\) is not
isolated in \(\Comp(G)[\inc(\e)]\).  Therefore \(K(G)=H_1.\)

Now let \(g\in L(G)\setminus\{\e\}.\)
Then \(g\notin\inc(\e)\): otherwise, taking \(x=g\) in the defining
condition for \(L(G)\) would require \(g\parallel g,\)
which is impossible.  Hence \(g\) is comparable with \(\e\), so its
value has type \(\Ctype\).

It cannot lie in \(H_1\), whose nonidentity elements have type
\(\Atype\).  Therefore \(v(g)>H_1.\)
For every \(h\in H_1\setminus\{\e\}\), the larger value \(v(g)\) has
type \(\Ctype\), so the leading-layer rule makes \(g\) comparable with
\(h\).  But \(h\in\inc(\e),\)
contradicting the defining condition of \(L(G)\).  Hence \(L(G)=\{\e\}.\)

Now suppose that \(H_1\) has type \(\Ctype\).  Let \(x\in\inc(\e).\)
Then \(x\notin H_1\), because every nonidentity element of \(H_1\) is
comparable with \(\e\).  Thus \(v(x)\) is a higher type \(\Atype\)
level.  Hence, for every \(h\in H_1\setminus\{\e\},\)
the larger value is \(v(x)\), of type \(\Atype\), and therefore \(h\parallel x.\)
Also \(\e\parallel x\)
by the definition of \(\inc(\e)\).

Let \(h\in H_1\setminus\{\e\}\).  Since there are at least two
canonical levels, \(H_2\) exists and, by alternation, has type
\(\Atype\).  Choose \(t\in H_2\setminus H_1.\)
Then \(v(t)=H_2\), so \(\e\parallel t.\)
Since \(v(h)=H_1<H_2=v(t)\), the leading-layer rule also gives \(t\parallel h.\)
Thus \(\e-t-h\)
is a path in \(\Inc(G)\), and hence \(h\in I(G)(\e).\)
We already observed that \(h\parallel x\) for every
\(x\in\inc(\e)\), so \(h\in L(G).\)
Together with \(\e\in L(G)\), this yields \(H_1\subseteq L(G).\)

Conversely, let \(z\notin H_1.\)
If \(z\in\inc(\e)\), then \(z\notin L(G)\), because taking \(x=z\)
in the defining universal condition would require \(z\parallel z\).

Suppose therefore that \(z\notin\inc(\e).\)
Then \(z\ne\e\) is comparable with \(\e\), so \(v(z)\) has type
\(\Ctype\).  Since \(z\notin H_1\), write \(v(z)=H_j\)
with \(j\geq2\).  By alternation, \(H_{j-1}\) has type \(\Atype\).
Choose \(t\in H_{j-1}\setminus H_{j-2},\)
where \(H_0=\{\e\}\) if \(j=2\).  Then \(v(t)=H_{j-1},\)
so \(t\in\inc(\e).\)
Because \(v(t)<v(z)\)
and \(v(z)\) has type \(\Ctype\), the leading-layer rule makes \(z\)
comparable with \(t\).  Hence \(z\notin L(G)\).

Therefore \(L(G)=H_1.\)

Finally, let \(x\in\inc(\e)\)
and choose \(\e\ne h\in H_1.\)
Since \(v(h)=H_1<v(x),\)
the leading-value law gives \(v(xh)=v(x).\)
The value \(v(x)\) has type \(\Atype\), so \(xh\in\inc(\e).\)
Moreover, \(x^{-1}(xh)=h,\)
and \(h\) is comparable with \(\e\).  Hence \(x\) is comparable with
\(xh\).  Since \(h\ne\e\), the two elements are distinct.  Thus no
element of \(\inc(\e)\) is isolated in
\(\Comp(G)[\inc(\e)]\).  Therefore \(K(G)=\{\e\}.\)
\end{proof}

The proposition identifies the least canonical level directly, without
first determining the canonical chain: exactly one of \(K(G)\) and
\(L(G)\) is nontrivial, and that set is \(H_1\).  Let \(B(G)\) denote
this nontrivial set.  Since \(H_1\) is a normal strong subgroup,
\cref{cor:quotient-compatibility} gives \(\cV(G/B(G)) = \{H_2/H_1,\ldots,H_r/H_1\}.\)
Thus, as long as at least two canonical levels remain, the proposition
recognizes the least one intrinsically and quotienting by it removes that
level.  Iterating this procedure determines all but the final canonical
level; when only one level remains, it is the whole remaining quotient.
Hence the finite canonical chain is reconstructed from the bottom up.

\section{Realisation and examples}\label{sec:examples}

\subsection{Canonical value chains}

The canonical value chain need not be finite or discrete.  We record first
that every chain admits the reduced two-colouring required by the structural
characterization.

\begin{lemma}\label{lem:reduced-colouring-chain}
Every nonempty chain admits a reduced two-colouring.
\end{lemma}

\begin{proof}
Let \(C\) be a nonempty chain.  Consider Rosenstein's finite
condensation \cite[Chapter~4, \S2]{Rosenstein1982}:
\[
 x\sim y
 \quad\Longleftrightarrow\quad
 [x,y]\text{ is finite}.
\]
Its classes are convex, and every nonsingleton class is finite or has
order type \(\omega\), \(\omega^*\), or \(\zeta\).  Hence every
nonsingleton class admits an alternating two-colouring.

Let \(D=C/{\sim}\), and let \(S\subseteq D\) be the set of singleton
\(\sim\)-classes.  Partition \(S\) into its maximal convex subsets.  If such
a subset \(J\) has at least two elements, then \(J\) is densely ordered.
Indeed, if two members of \(J\) were consecutive in \(J\), convexity of
\(J\) in \(D\) would make them consecutive in \(D\); the corresponding two
points of \(C\) would then have a finite interval between them and hence
belong to the same finite-condensation class, a contradiction.  By
\cite[Theorem~21.1]{Monk2024}, every such \(J\) is the disjoint union of two
dense subsets; use these as the two colour classes on \(J\).  Colour an
isolated singleton member of \(S\) arbitrarily.

We claim that the resulting colouring of \(C\) is reduced.  Suppose, to the
contrary, that \(x<y\) have the same colour and that there is no point of the
opposite colour strictly between them.  If \(x\sim y\), the alternating
colouring of their condensation class gives such a point, so their
condensation classes are distinct.

Every condensation class strictly between the classes of \(x\) and \(y\)
must be a singleton: a nonsingleton class contains points of both colours
and lies entirely between \(x\) and \(y\).  Moreover, if the class of \(x\)
is nonsingleton, then \(x\) must be its greatest element, since otherwise
the next point of that class above \(x\) has the opposite colour.  Dually,
if the class of \(y\) is nonsingleton, then \(y\) must be its least element.

There cannot be zero or one singleton condensation classes strictly between
the classes of \(x\) and \(y\).  Under the endpoint conclusion just obtained,
the interval \([x,y]\) would then be finite, forcing \(x\sim y\), a
contradiction.  Hence there are at least two singleton condensation classes
strictly between them.  These classes lie in one maximal convex subset
\(J\subseteq S\), and \(J\) has at least two elements.  Since both colour
classes are dense in \(J\), there is a point of the opposite colour between
\(x\) and \(y\), again a contradiction.  Therefore the colouring is reduced.
\end{proof}

\begin{theorem}[Realisation of canonical value chains]
\label{thm:realisation}
Every nonempty reduced two-coloured chain \((\Gamma,\lambda)\) is isomorphic, as a
coloured chain, to the canonical value chain of an \(N\)-free ordered
abelian group whose canonical factors are all isomorphic as groups to
\(\mathbb Z\).  Consequently every nonempty chain occurs as \(\cV(G)\) for
some \(N\)-free ordered abelian group \(G\).
\end{theorem}

\begin{proof}
Let \((\Gamma,<,\lambda)\) be a nonempty reduced two-coloured chain.  We use additive notation and let
\[
G_\Gamma:=\bigoplus_{\gamma\in\Gamma}\mathbb Z e_\gamma
\]
be the free abelian group with basis
\(\{e_\gamma:\gamma\in\Gamma\}\).
Thus every \(g\in G_\Gamma\) has a unique representation \(g=\sum_{\gamma\in F} n_\gamma e_\gamma,\)
where \(F\subseteq\Gamma\) is finite, and we write \(\supp(g):=\{\gamma\in\Gamma:n_\gamma\neq0\}.\)
For \(g\neq0\), define \(v(g):=\max\supp(g),\)
which exists because the support is finite, and define
\[
 H_\gamma:=\{g:\supp(g)\subseteq(-\infty,\gamma]\}.
\]
If \(v(g)=\gamma\), then \(H_\gamma\) is the least member of the chain
\(\{H_\delta\}\) containing \(g\), and
\[
 H_\gamma^-=
 \{g:\supp(g)\subseteq(-\infty,\gamma)\},
 \qquad
 H_\gamma/H_\gamma^-\cong\mathbb Z.
\]
Give this quotient the usual order when \(\lambda(\gamma)=\Ctype\) and the
equality order when \(\lambda(\gamma)=\Atype\).  Since \(G_\Gamma\) is
abelian, each \(H_\gamma^-\) is normal in \(H_\gamma\), and the
conjugation condition is automatic.  Thus \(\{H_\gamma:\gamma\in\Gamma\}\) is a reduced admissible
two-coloured subgroup chain.  By \cref{thm:leading-representation}, the
leading-layer rule defines an \(N\)-free ordered abelian group and \(S(0,g)=H_{v(g)}\qquad(0\ne g\in G_\Gamma).\)
Hence its canonical chain is isomorphic to \((\Gamma,\lambda)\), and every
canonical factor is isomorphic as a group to \(\mathbb Z\).

For an arbitrary nonempty chain \(C\), choose a reduced two-colouring by
\cref{lem:reduced-colouring-chain} and apply the construction above.  This
proves the final assertion as well.
\end{proof}

Thus dense and nondiscrete canonical chains occur.  For instance one may
take \(\Gamma=\mathbb Q\) with two dense colour classes, or
\(\Gamma=\mathbb R\) with the rational points of type \(\Ctype\) and the
irrational points of type \(\Atype\).  By
\cref{thm:intro-gallai}\textup{(a)}, an initial segment with no largest
canonical value determines a strong subgroup which need not itself be the
value of any element.

\begin{example}
\label{ex:limit-collapse}
Take \(\Gamma=\mathbb R\), colour \(\mathbb Q\) by \(\Ctype\) and
\(\mathbb R\setminus\mathbb Q\) by \(\Atype\), and form the finite-support
realisation above.  Let
\[
 K:=\bigoplus_{q\in\mathbb Q}\mathbb Z e_q\leq G_\Gamma.
\]
The active ambient canonical levels of \(K\) are exactly the rational levels,
and all have type \(\Ctype\).  Hence they form one maximal convex
same-type block.  The subgroup-restriction theorem therefore gives
\[
 \cV(K)=\{K\}.
\]
In particular, the induced order on \(K\) is total.  Since the rational
levels are cofinal in the ambient chain, for every
\(\alpha\in\mathbb R\) there is \(q\in\mathbb Q\) with \(q>\alpha\), and
hence \(e_q\in K\setminus H_\alpha\).  Thus the unique canonical level
\(K\) of the induced order is not \(K\cap H_\alpha\) for any single
ambient canonical level \(H_\alpha\); rather, it is the union of the
increasing family of such intersections.  This is the limit phenomenon
which forces the block-union formulation in \cref{thm:intro-subgroup-restriction}.
\end{example}

\begin{example}[Weak orders]
Recall that a weak order is a linear sum of antichains.  Pouzet and Zaguia
\cite[Proposition~2]{PouzetZaguia2019} prove the following.

\begin{proposition}\label{claim:weakorder}
Let \(G\) be an ordered group and \(\inc(\e)\) be the set of elements
incomparable to \(\e\).  The following assertions are equivalent:
\begin{enumerate}[label=\textup{(\roman*)}]
\item The order of \(G\) is a weak order.
\item \(\inc(\e)\) is an antichain.
\item \(\inc(\e)\cup\{\e\}\) is a subgroup of \(G\).
\item \(\inc(\e)\cup\{\e\}\) is a module.
\end{enumerate}
If any of the above conditions hold, then \(H:=\inc(\e)\cup\{\e\}\) is a
normal subgroup of \(G\), the quotient group \(G/H\) is totally ordered
and the order on \(G\) is the lexicographical sum of copies of \(H\)
indexed by \(G/H\).
\end{proposition}

For example, if \(F_r\) is a free group of rank \(r\ge2\) and
\(\varphi:F_r\to\mathbb Z\) is an epimorphism, order the
\(\ker\varphi\)-cosets according to the usual order on \(\mathbb Z\) and
make each coset an antichain.  This gives a compatible weak order on
\(F_r\) which is neither total nor equality.
\end{example}

\begin{example}[An \(N\)-free order which is not a weak order on \(F_r\)]
Let \(F_r\) be a free group of rank \(r\geq 2\).  Since \(F_r\) is
bi-orderable \cite{Glass1999}, fix a bi-invariant total order \(\leq\) on
\(F_r\).

Choose a proper nontrivial normal subgroup \(H\triangleleft F_r\).  For
example, let
\[
\varphi:F_r\longrightarrow \mathbb Z_2
\]
be an epimorphism and put
\[
H:=\ker\varphi.
\]

Define a partial order \(\preceq\) on \(F_r\) by
\[
x\preceq y
\quad\Longleftrightarrow\quad
xH=yH
\ \text{ and }\
x\leq y.
\]
Thus each coset of \(H\) is totally ordered by the restriction of
\(\leq\), while elements belonging to distinct cosets of \(H\) are
incomparable.

The order \(\preceq\) is compatible with the group operation.  Indeed, if
\(x\preceq y\), then \(xH=yH\) and \(x\leq y\).  Hence, for every
\(a,b\in F_r\),
\[
axb\,H=ayb\,H,
\]
because \(H\triangleleft F_r\), and
\[
axb\leq ayb
\]
because \(\leq\) is bi-invariant.  Therefore
\[
axb\preceq ayb.
\]

The poset \((F_r,\preceq)\) is \(N\)-free.  Its comparability graph is
the disjoint union of complete graphs induced by the cosets of \(H\),
and hence contains no induced \(P_4\).

On the other hand, \((F_r,\preceq)\) is not a weak order.  Choose
\(h\in H\setminus\{\e\}\) with
\[
\e<h
\]
in the fixed bi-order, and choose \(g\notin H\).  Then
\[
\e\prec h,
\qquad
g\parallel \e,
\qquad
g\parallel h.
\]
Thus incomparability is not transitive, so \((F_r,\preceq)\) is not a
weak order.

Equivalently, the order is the lexicographical sum
\[
F_r
=
\sum_{F_r/H}^{\mathrm{equality}} H,
\]
where each copy of \(H\) carries the restricted total order.
\end{example}

\section*{Declaration of generative AI and AI-assisted technologies in the writing process}
During the preparation of this work, the author used ChatGPT (OpenAI)
to assist with language editing, organization, and the presentation of
mathematical arguments.  The author independently verified all
mathematical statements, proofs, references, and conclusions, reviewed
and edited the resulting text as needed, and takes full responsibility
for the content of the article.

\end{document}